\documentclass[12pt, twoside]{article}
\usepackage{amsmath,amsthm,amssymb}
\usepackage{times}
\usepackage{enumerate}
\usepackage{extarrows}
\usepackage{color}
\usepackage{tabularx}
\usepackage{multirow}
\usepackage{graphicx}

\def\titlerunning#1{\gdef\titrun{#1}}
\makeatletter
\def\author#1{\gdef\autrun{\def\and{\unskip, }#1}\gdef\@author{#1}}
\def\address#1{{\def\and{\\\hspace*{18pt}}\renewcommand{\thefootnote}{}
\footnote {#1}}
\markboth{\autrun}{\titrun}}
\makeatother
\def\email#1{e-mail: #1}

\newtheorem{theorem}{Theorem}[section]
\newtheorem{lemma}[theorem]{Lemma}
\newtheorem{definition}[theorem]{Definition}
\newtheorem{proposition}[theorem]{Proposition}
\newtheorem{remark}[theorem]{Remark}
\newtheorem{corollary}[theorem]{Corollary}

\newtheorem{example}[theorem]{Example}

\newtheorem{TheoremX}{Theorem}

\newcommand{\gm}{\gamma}

\newcommand{\R}{\mathbb{R}}
\newcommand{\cR}{\mathcal{R}}
\newcommand{\cS}{\mathcal{S}}
\newcommand{\cG}{\mathcal{G}}
\newcommand{\cN}{\mathcal{N}}
\newcommand{\eps}{\varepsilon}
\newcommand{\Proof}{\begin{proof}}
\newcommand{\End}{\end{proof}}

\numberwithin{equation}{section}

\newcommand{\PreserveBackslash}[1]{\let\temp=\\#1\let\\=\temp}
\newcolumntype{C}[1]{>{\PreserveBackslash\centering}p{#1}}
\newcolumntype{R}[1]{>{\PreserveBackslash\raggedleft}p{#1}}
\newcolumntype{L}[1]{>{\PreserveBackslash\raggedright}p{#1}}

\newcolumntype{I}{!{\vrule width 1pt}}
\newlength\savedwidth

\begin{document}


\baselineskip=15pt


\titlerunning{On the dynamics of contact Hamiltonian systems II:\\ Variational construction of asymptotic orbits}

\title{On the dynamics of contact Hamiltonian systems II:\\ Variational construction of asymptotic orbits}

\author{Liang Jin  \and Jun Yan \and Kai Zhao}

\date{\today}

\maketitle

\address{Liang Jin: Department of Mathematics, Nanjing University of Science and Technology, Nanjing 210094, China; Fakult\"{a}t f\"{u}r Mathematik, Ruhr-Universit\"{a}t Bochum, Universit\"{a}tstra$\beta$e 150, D-44801 Bochum, Germany;\,\,\email{jl@njust.edu.cn;\,\, Liang.Jin@ruhr-uni-bochum.de}
\and Jun Yan: School of Mathematical Sciences, Fudan University, Shanghai 200433, China;\,\,\email{yanjun@fudan.edu.cn}
\and Kai Zhao:  School of Mathematical Sciences, Tongji University, Shanghai 200092, China;\,\,\email{zhaokai93@tongji.edu.cn}}

\begin{abstract}
  This paper is a continuation of our study of the dynamics of contact Hamiltonian systems in \cite{JY}, but without monotonicity assumption. Due to the complexity of general cases, we focus on the behavior of action minimizing orbits. We pick out certain action minimizing invariant sets $\{\widetilde{\mathcal{N}}_u\}$ in the phase space naturally stratified by solutions $u$ to the corresponding Hamilton-Jacobi equation. Using an extension of characteristic method, we establish the existence of semi-infinite orbits that is asymptotic to some $\widetilde{\mathcal{N}}_u$ and heteroclinic orbits between $\widetilde{\mathcal{N}}_u$ and $\widetilde{\mathcal{N}}_v$ for two different solutions $u$ and $v$.
\end{abstract}

\newpage

\tableofcontents

\newpage

\section{Introduction and main results}
\setcounter{equation}{0}
\setcounter{footnote}{0}
Let $M$ be a closed, connected smooth manifold. As usual, $J^0(M,\R)=M\times\R$ and $J^1(M,\R)=T^{\ast}M\times\R$ denote the manifold of 0-jets and 1-jets of functions on $M$. We adopt the convention in \cite{JY} to use both $z$ and $(x,u,p)$ to denote points in $J^1(M,\R)$, where $(x,p)$ is the Liouville coordinates on the cotangent bundle of $M$ and $u\in\R$. Let $\rho:J^1(M,\R)\rightarrow M$ stands for the natural projection $(x,u,p)\mapsto x$. Equipped with the standard contact structure $\xi_{std}$ defined by the kernel of the $1$-form $du-pdx, (J^1(M,\R), \xi_{std})$ becomes a classical example of contact manifold and $\pi:J^1(M,\R)\rightarrow J^0(M,\R); (x,u,p)\mapsto(x,u)$ denotes the canonical front projection. We consider the ODE system generated by a $C^{\infty}$ function $H:J^1(M,\R)\rightarrow\R$, called  \textbf{contact Hamiltonian}, via the vector field
\begin{equation}\label{ch}
X_H=
\begin{cases}
\dot{x}=\partial_p H(x,u,p),\\
\dot{p}=-\partial_x H(x,u,p)-\partial_u H(x,u,p) p,\\
\dot{u}=\partial_p H(x,u,p)\cdot p-H(x,u,p),
\end{cases}
\end{equation}
and its local flow $\Phi^t_H$, which consist of contactmorphisms from $(J^1(M,\R),\xi_{std})$ to itself. Once and for all, we fix an auxiliary Riemannian metric $g$ on $M$. In the following context, we shall assume the contact Hamiltonian $H$ satisfies
\begin{enumerate}
	\item[(H1)] (Fiberwise convexity)\quad for any $(x,u,p)\in J^1(M,\R)$, the Hessian
                 \[
                 \partial^2_{pp}H(x,u,p): T_p(T^\ast_x M)\times T_p(T^\ast_x M)\rightarrow\R,
                 \]
                 as a quadratic form, is positive definite;
	
    \item[(H2)] (Fiberwise superlinearity)\quad for any $(x,u)\in M\times\R$ and $K>0$, there exists $C^{\ast}(K)>0$ such that
                 \[
                 H(x,u,p)\geqslant K\|p\|_x-C^{\ast}(K);
                 \]
                 for any $p\in T^\ast_xM$. Here, the norm $\|\cdot\|_x$ is induced by the metric $g$,
\end{enumerate}
which we shall refer as \textbf{Tonelli conditions} later. Denote by $\mathcal{R}=\frac{\partial}{\partial u}$ the \textbf{Reeb vector field} defined by the 1-form $du-pdx$ and by $\omega(z)($ resp. $\alpha(z))$ the $\omega$-limit (resp. $\alpha$-limit) set of $z$ under $\Phi^t_H$.

\subsection{Previous work on the dynamics of the monotone systems}
In \cite{JY}, we began our study of system \eqref{ch} by understanding the global dynamics of $\Phi^t_H$ under either of the monotonicity assumptions
\begin{enumerate}
  \item [(M$_+$)] there is $\kappa>0$ such that $dH(\mathcal{R})>\kappa$ on the whole phase space $J^1(M,\R)$.
  \item [(M$_-$)] there is $\kappa>0$ such that $dH(\mathcal{R})<-\kappa$ on the whole phase space $J^1(M,\R)$.
\end{enumerate}
Motivated by the previous works \cite{MS} in the context of conformally symplectic systems, and  \cite{WWY3} in the context of monotone contact Hamiltonian systems, we are able to show
\begin{theorem}\label{mono}\cite[Page 3314, Theorem A]{JY}
Assume $H$ satisfies (M$_+$), then
\begin{enumerate}[(1)]
  \item all the orbits are forward complete, i.e., all $\Phi^t_H$-orbits are defined for $t\in[0,+\infty)$;
  \item $\Phi^t_H$ admits a compact, connected invariant set $\mathcal{K}_H\subset H^{-1}(0)$ such that $\omega(z)$ is contained in $\mathcal{K}_H$ for each $z\in J^1(M,\R)$ and $\alpha(z)\neq\emptyset$ if and only if $z\in\mathcal{K}_H$. $\mathcal{K}_H$ is called a maximal global attractor for $\Phi^t_H$ and equals the union of all compact invariant sets.
  \item $\mathcal{K}_H$ admits a neighborhood basis $\{\mathcal{O}_t\}_{t\geqslant0}$ with each element $\mathcal{O}_t$ is homotopy equivalent to $M$.
\end{enumerate}
By reversing the time direction, all conclusions above have analogies when (M$_-$) holds.
\end{theorem}

To prove the proposition, we employ two Lyapunov functions for the flow $\Phi^t_H$ to show that all compact invariant sets is contained in a common compact subset of the phase space. Notice that the second Lyapunov function, related to the viscosity solution to the Hamilton-Jacobi equation
\begin{equation}\label{hj}\tag{HJ}
H(x,u,d_xu)=0
\end{equation}
is of particular interest. Due to (M$_+$), the solution $u_-$ to \eqref{hj} is \textbf{unique} but in general non-smooth. With the help of Tonelli conditions, $u_-$ has directional derivative and the function $F(z)=u_-(x)-u$ serves as Lyapunov functions on that part of the phase space where $F$ is non-negative. As a result, the maximal global attractor $\mathcal{K}_H\subset\{(x,u,p)\in J^1(M,\R):u\geqslant u_-(x)\}$. And the pre-image of the graph of $u_-$ under $\pi$ intersects $\mathcal{K}_H$ in a compact $\Phi^t_H$-invariant set $\widetilde{\cN}_{u_-}$, which is regarded as the upper frontier of the maximal global attractor.

\vspace{0.5em}
One of the main result, Theorem A, of this paper aims to convince the readers that the set $\widetilde{\cN}_{u_-}$ can play the role of the attractor for certain action minimizing orbits rather than all orbits, in the study of the dynamics of $\Phi^t_H$ \textbf{without monotonicity assumptions}.

\vspace{.5em}
Introducing the further assumption
\begin{enumerate}
	\item[($\rm{F}$)] $\partial_p H(x,u,0)=0$ for every $(x,u)\in M\times\R$; let
    \[
    \mathcal{F}_H:=\{(x_0,u_0,0)\in J^1(M,\R):H(x_0,u_0,0)=0, \partial_x H(x_0,u_0,0)=0\}
    \]
    be the set of fixed points of $\Phi^t_H$, for each $(x_0,u_0,0)\in\mathcal{F}_H$, $x_0$ is non-degenerate critical point of the function $H(\cdot,u_0,0)$. The last condition ensures that the fixed points of $\Phi^t_H$ are isolated.
\end{enumerate}

we proved that

\begin{theorem}\label{irre}
Assume $H$ satisfies $(\rm{F})$, then
\begin{enumerate}[(1)]
  \item The maximal global attractor $\mathcal{K}_H$ consists of finite fixed points and heteroclinic orbits between them.
  \item  A necessary condition for the existence of $z\in J^1(M,\R)$ between two fixed points $z_0=(x_0,u_0,0)$ and $z_1=(x_1,u_1,0)$ with $\alpha(z)=z_0$ and $\omega(z)=z_1$ is $u_0<u_1$.
\end{enumerate}
\end{theorem}

Here heteroclinic orbits between fixed points lie on $H^{-1}(0)$ and their existence is \textbf{guaranteed by the topological information} of $\mathcal{K}_H$. In our second main theorem, Theorem B, we constructed heteroclinic orbits between compact invariant sets for more general contact Hamiltonian flows by variational methods. However, these orbits \textbf{may not lie on the null energy level}. The second conclusion is a simple criterion for the orbital transition inside $\mathcal{K}_H$. Despite the observable proof of the theorem, it reflects some kind of \textbf{irreversibility of $\Phi^t_H$} and is \textbf{by no means exceptional}. Similar phenomena are discovered in (B1).

\vspace{0.5em}
Careful readers may find the condition (H2) employed in \cite{JY} is slightly weaker than the one we used here, the reason is purely methodological. One can refer to \cite{JY} for more details. For those readers who are interested in the variational aspects of monotone contact Hamiltonian flows, we recommend \cite{MS},\cite{WWY3}.

\subsection{General assumptions on contact Hamiltonian and a toy model}
In this subsequent note, we continue to explore the system \eqref{ch} and $\Phi^t_H$ under (H1)-(H2) and the more general assumption
\begin{enumerate}	
    \item[(H3)] (Uniformly Lipschitz in $u$)\quad there exists $L>0$ such that
                 \[
                 |H(x,u,p)-H(x,u',p)|\leqslant L|u-u'|
                 \]
                 for any pair of points $(x,u,p), (x,u',p)$.
\end{enumerate}
Notice that when the contact Hamiltonian does not depend on $u$, (H3) is automatically satisfied. Thus the set of classical Hamiltonian also falls into our consideration. In this case, the last equation of \eqref{ch} is independent of the remaining ones and $\Phi^t_H$ is a suspension of a Hamiltonian flow. The additional variable $u$ records the increment of action along orbits of the Hamiltonian flow. Although the volume of phase space is still preserved, the dynamics of this suspension flow maybe very different. For instance, there is no periodic orbits for $\Phi^t_H$ outside $H^{-1}(0)$. One may expect that the contact suspension of Hamiltonian flow offers us a new way to rethinking and reformulating theorems in Hamiltonian dynamics.

\vspace{0.5em}
Another reason to introduce (H3) is that the main techniques we employed in this work coming from the variational framework for $\Phi^t_H$ developed in \cite{WWY1}-\cite{WWY3}. In this series of works, (H1)-(H3) are standard assumptions  and are used to deduce the implicit action functions introduced in Proposition \ref{Implicit variational}, which are fundamental tools for our study. It would be interesting to see how we can weaken these assumptions as well as the compactness of $M$ to make our results applicable to more physical models.

\vspace{0.5em}
We also want to remark that the theory built in \cite{WWY1}-\cite{WWY3} can be regarded as an extension of the celebrated Aubry-Mather theory for Tonelli Hamiltonian systems, convincing us that (H1)-(H2) are natural. For the readers who are interested in the Aubry-Mather theory, we recommend the original works \cite{M1}-\cite{M2} of John.Mather and the standard references \cite{Fathi_book},\cite{S},\cite{CI} for elegant elaborations.

\vspace{0.5em}
For the readers who want a concrete class of contact Hamiltonian satisfying (H1)-(H3) and including all cases we have discussed so far, let us consider
\begin{example}\label{toy-model}\cite{JYZ}
The contact Hamiltonian $H:J^1(M,\R)\rightarrow\R$ has the following decoupled form
\begin{equation}\label{model1}
H(x,u,p)=F(x,p)+\lambda(x)u,
\end{equation}
where $F$ satisfies (H1)-(H2) and $\lambda\in C^\infty(M,\R)$. It is easy to check that
\begin{equation}\label{model-mono}
dH(\cR)=\lambda(x)
\end{equation}
and $H$ satisfies all assumptions by  taking $L$ appeared in (H3) to be $\max_{x\in M}|\lambda(x)|$. If $\lambda(x)$ is no-where vanishing on $M$, then there exists $\kappa>0$ such that
\begin{enumerate}
  \item[$(+)$] $\lambda(x)\geqslant\kappa>0$, thus by \eqref{model-mono}, $H$ satisfies (M$_+$);
  \item[$(-)$] $\lambda(x)\leqslant-\kappa<0$, thus by \eqref{model-mono}, $H$ satisfies (M$_-$).
\end{enumerate}
If $\lambda(x)\equiv0$, then $H=F(x,p)$ is a classical Tonelli Hamiltonian.

\vspace{0.5em}
On the other hand, the set of model Hamiltonian \eqref{model1} is \textbf{larger} than the types of Hamiltonian mentioned above: when $\lambda$ \textbf{change signs} on $M$, then by \eqref{model-mono}, (M$_\pm$) are violated and $H$ depends the action variable $u$ genuinely. As a result, the structure of the solution set to \eqref{hj} is more complicated than monotone cases.
\end{example}

In Section 6, we shall summarize some research history on the dynamics of $\Phi^t_H$, emphasising on the case with positive $\lambda(x)$, and then apply our main theorems to give some news on the last \textbf{non-monotone} case. The statement of the main theorems involves some terminologies which we introduce in the next part.

\subsection{Solution semigroups and weak KAM solutions for Hamilton-Jacobi equations}
Since \eqref{ch} can be interpreted as the characteristic system for Hamilton-Jacobi equation \eqref{hj}, it is natural to predict the close relationship between dynamics of $\Phi^t_H$ and the solution of the associated PDEs. In this section, we shall introduce analytic tools to understand the solutions to the Cauchy problem
\begin{equation}\label{HJe}\tag{HJe}
\begin{cases}
\partial_t U(x,t)+G(x,U(x,t),\partial_x U(x,t))=0,\quad\, (x,t)\in M\times(0,+\infty),\\
\hspace{11.4em}U(x,0)\,=\varphi(x),\hspace{1.3em}x\in M,
\end{cases}
\end{equation}
and the corresponding stationary problem
\begin{equation}\label{HJs}\tag{HJs}
G(x,u(x),d_xu(x))=0,\quad x\in M.
\end{equation}
Here and after, $G$ equals $H$ or $\breve{H}(x,u,p)=H(x,-u,-p)$, also satisfying (H1)-(H3).

\vspace{0.5em}
It is well known that, in general, \eqref{HJe} and \eqref{HJs} do not admit $C^1$ solutions. Thus to treat the above equations, it is necessary to extend the notion of `solutions' to include non-smooth functions. Fortunately, the right one, called \textbf{viscosity solution}, was introduced by M.Crandall and P.L.Lions in \cite{CL}. It is now widely accepted as the natural framework for the theory of Hamilton-Jacobi equations and certain second order PDEs. We refer to \cite{CS},\cite{Evans-PDE},\cite{Ishii_chapter} and \cite{CIL} for nice expositions of the theory. In the following context, solutions to \eqref{HJs} and \eqref{HJe} are always understood in the viscosity sense. Under assumptions (H1)-(H3), it was shown in \cite{WWY2} that there are two families of one-parameter nonlinear operators $\{T_t^\pm\}_{t\geqslant0}: C(M,\R)\circlearrowleft$ satisfying
\begin{itemize}
	\item [(1)] $T_0^\pm\varphi=\varphi$ for all $\varphi\in C(M,\R)$;
	\item [(2)] $T_{t+s}^\pm\varphi= T_t^\pm(T_s^\pm\varphi) $ for all $t,s \geqslant 0$ for all $\varphi\in C(M,\R)$;
	\item [(3)] $T_t^\pm\varphi $ is continuous in $(t,\varphi)$  on $[0,+\infty) \times C(M,\R)$.
\end{itemize}
such that

\begin{proposition}\label{sol-HJe}
For each $\varphi\in C(M,\R)$, the function $U:M\times[0,+\infty)\rightarrow\R$,
\[
(x,t)\mapsto U(x,t):=T_t^-\varphi(x)\,\,\big(\text{resp. }-T_t^+(-\varphi)(x)\big)
\]
is the \textbf{unique viscosity solution} of equation of \eqref{HJe} with $G=H$ (resp. $G=\breve{H}$).
\end{proposition}
We call $\{T^{-}_t\}_{t\geqslant 0}$ and $\{T^{+}_t\}_{t\geqslant 0}$ the \textbf{backward and forward solution semigroup} associated to the Cauchy problem \eqref{HJe} with $G=H$.

\begin{definition}\label{fp}
A continuous function $u_{-}\,\,($resp. $u_+):M\rightarrow\R$ is called a backward (resp. forward) weak KAM solution to \eqref{hj} if it is a fixed point of $\{T^{-}_t\}_{t\geqslant 0}\,\,\big($resp. $\{T^{+}_t\}_{t\geqslant 0}\big)$, i.e., for any $t\geq0$,
\[
T^{-}_t u_-=u_-\quad(\text{resp.}\,\,T^{+}_t u_+=u_+).
\]
We use $\mathcal{S}_-$ and $\mathcal{S}_+$ to denote the set of backward (resp. forward) weak KAM solutions to \eqref{hj} respectively.
\end{definition}

The relation between weak KAM solutions and viscosity solutions to \eqref{HJs} is included in
\begin{proposition}\cite[Proposition 2.8]{WWY3}\label{weak-kam-vis}
Using the notations defined above,
\begin{enumerate}[(1)]
  \item $u_-\in\mathcal{S}_-$ if and only if $\,\,\,\,u_-$ is a solution to \eqref{HJs} with $G=H$,
  \item $u_+\in\mathcal{S}_+$ if and only if $-u_+$ is a solution to \eqref{HJs} with $G=\breve{H}$.
\end{enumerate}
\end{proposition}

\vspace{1em}
Recall that $p\in T^\ast_xM$ is called a \textbf{reachable differential} of a continuous function $f:M\rightarrow\R$ at $x$ if there is a sequence $\{x_n\}\subset M$ converges (in $M$) to $x$  such that $f$ is differential at each $x_n$ and the sequence $\{p_n=d_x f(x_n)\}$ converges to $p$ (in $T^{\ast}M$). We use $D^\ast f(x)$ to denote the set of all reachable differentials of $f$ at $x$.

\begin{definition}\label{pseudo-gr}
For a Lipschitz function $f:M\rightarrow\R$, we introduce the notions:
\begin{itemize}
  \item $0$-graph of $f:\quad\mathrm{J}^0_f:=\{(x,f(x))\,:\,x\in M\}$,

  \item $1$-pseudograph of $f:$
        \begin{equation}
        \begin{split}
        \mathrm{J}^1_f:=&\{(x,u,p)\in J^1(M,\R):u=f(x),\,\, p\in D^\ast f(x)\}\\
        =&\overline{\{(x,u,p)\in J^1(M,\R):f\text{ is differentiable at } x,u=f(x), p=d_xf(x)\}}.
        \end{split}
        \end{equation}
\end{itemize}
It follows that $\mathrm{J}^0_f$ and $\mathrm{J}^1_f$ are compact subsets of $J^0(M,\R)$ and $J^1(M,\R)$ respectively. Rademacher's theorem states that a Lipschitz function is differentiable almost everywhere (with respect to the Riemannian volume on $M$), thus $\rho\mathrm{J}^1_f=M$. For $f\in C^2(M,\R), \mathrm{J}^1_f$ is a Legendrian graph over $\rho$.
\end{definition}

\vspace{0.5em}
As we shall see later, weak KAM solutions are Lipschitz functions on $M$. Thus it makes sense to speak about their $1$-pseudographs. Furthermore, due to Corollary \ref{inv-graph} later, the $1$-pseudograph of a backward (resp. forward) weak KAM solution is backward (resp. forward) invariant under $\Phi_H^t$. Thus it is natural to introduce

\begin{definition}\label{Mane-slice}
For any $u_-\in\mathcal{S_-}$, we call the compact $\Phi_H^t$-invariant set
\[
\widetilde{\mathcal{N}}_{u_-}=\bigcap_{t\geqslant 0}\Phi_H^{-t}(\mathrm{J}^1_{u_-})
\]
the Ma\~n\'e slice associated with $u_-$. In a similar fashion, for any $u_+\in\mathcal{S_+}$, the compact $\Phi_H^t$-invariant subset
\[
\widetilde{\mathcal{N}}_{u_+}=\bigcap_{t\geqslant 0}\Phi_H^{t}(\mathrm{J}^1_{u_+})
\]
is called the Ma\~n\'e slice associated with $u_+$.
\end{definition}

\begin{remark}
Thus for each weak KAM solution $u_{\pm}$, there associates a unique compact $\Phi^t_H$-invariant set $\widetilde{\mathcal{N}}_{u_\pm}$. The orbits contained $\widetilde{\mathcal{N}}_{u_\pm}$ are action minimizing orbits defined by the variational principles in \cite{WWY1}-\cite{WWY3} as well as \cite{CCJWY}. And even more, they are time-free minimizers, see Section 5 for rigorous formulations.
\end{remark}

\subsection{Statement of the main results}
For a general contact Hamiltonian satisfying (H1)-(H3), we shall, instead of having a universal study of all $\Phi^t_H$-orbits, focus on the dynamics of certain \textbf{action minimizing orbits}. Due to a global version of characteristic theory for non-linear first order PDEs, these orbits are closely related to the solutions to corresponding equations \eqref{HJe} and \eqref{HJs}, of which we are able to, in several aspects, give a more detailed study.

\vspace{0.5em}
Roughly speaking, the main theorems of this paper give us samples about how to deduce news of the asymptotic behavior of action minimizing $\Phi^t_H$-orbits from the information on the asymptotic behavior of solutions to \eqref{HJe}. This kind of duality already appeared in the work of the weak KAM theory for classical Tonelli Hamiltonian systems, see \cite{Fathi_book}. The first theorem ensures the existence of semi-infinite orbits asymptotic to some Ma\~{n}\'{e} slice if the solution semigroup converges.

\begin{TheoremX}\label{thm1}
Assume for $\varphi\in C^2(M,\R)$, there exists $u_-\in\mathcal{S}_-$ such that
\begin{equation}\label{convergence1}
\lim_{t\rightarrow+\infty}T^-_t\varphi(x)=u_-(x)
\end{equation}
holds uniformly for all $x\in M$, then
\begin{itemize}
  \item[\rm{(A1)}] there is $Z\in\mathrm{J}^1_{\varphi}$ such that $\omega(Z)\subset\widetilde{\mathcal{N}}_{u_-}$,
  \item[\rm{(A2)}] $\mathrm{J}^1_{u_-}\subset\bigcap_{T>0}\big(\overline{\bigcup_{t\geqslant T}\Phi^t_H(\mathrm{J}^1_{\varphi})}\big)$.
\end{itemize}
\end{TheoremX}

\begin{remark}
It suffices to assume that \eqref{convergence1} holds pointwise.
\end{remark}

There is a general philosophy in our study of $\Phi^t_H$: any result concerning only the behavior of $\{T^-_t\}_{t\geqslant0}$ has an analogy by reversing the time direction. The following remark is an example:
\begin{remark}\label{time-reverse-thm1}
Assume for $\varphi\in C^2(M,\R)$, there is $u_+\in\mathcal{S}_+$ such that for all $x\in M$,
\begin{equation}\label{convergence2}
\limsup_{t\rightarrow+\infty}T^+_t\varphi(x)=u_+(x),
\end{equation}
then
\begin{itemize}
  \item[\rm{(A1')}] there is $Z\in\mathrm{J}^1_{\varphi}$ such that $\alpha(Z)\subset\widetilde{\mathcal{N}}_{u_+}$,
  \item[\rm{(A2')}] $\mathrm{J}^1_{u_+}\subset\bigcap_{T>0}\big(\overline{\bigcup_{t\geqslant T}\Phi^{-t}_H(\mathrm{J}^1_{\varphi})}\big)$.
\end{itemize}
\end{remark}

Given the convergence of \textbf{both} backward and forward solution semigroups on the \textbf{same} initial data, it is possible to deduce the existence of heteroclinic orbits connecting two $\Phi^t_H$-invariant sets described by Definition \ref{Mane-slice}.

\begin{TheoremX}\label{thm2}
Assume there is $\varphi\in C(M,\R)$  such that
\begin{equation}\label{eq:condition-thm2}
\quad\quad\lim_{t\to +\infty}T_t^- \varphi:=u_-, \quad  \lim_{t\to +\infty} T_t^+ \varphi:= v_+,\quad\text{and}\quad v_+ < u_-,
\end{equation}
where $u_-\in\mathcal{S}_-, v_+\in\mathcal{S}_+$ are weak KAM solutions. Then
\begin{itemize}
  \item[\rm{(B1)}] For any $z\in J^1(M,\R)$ with $\alpha(z)\cap\widetilde{\mathcal{N}}_{u_-}\neq\emptyset$, then $\omega(z)\cap\widetilde{\mathcal{N}}_{v_+}=\emptyset$.

  \item[\rm{(B2)}] there exists $Z\in J^1(M,\R)$ such that $\alpha(Z)\subset\widetilde {\mathcal{N}}_{v_+}, \omega(Z)\subset\widetilde {\mathcal{N}}_{u_-}$.
\end{itemize}
\end{TheoremX}

\begin{remark}
(B1) suggests a similar irreversible behavior of $\Phi^t_H$ as Theorem \ref{irre} (2) under the more general assumptions. Here fixed points in maximal global attractor of monotone systems are replaced by Ma\~{n}\'{e} slices. In both cases, the orbits are forced to travel with positive increment in the action variable $u$.
\end{remark}

\begin{remark}
The assumption $v_+<u_-$ guarantees that $\widetilde {\mathcal{N}}_{v_+}\cap\widetilde {\mathcal{N}}_{u_-}=\emptyset$, thus (B2) is non-trivial. In classical Hamiltonian system, energy conservation ensures that these heteroclinic orbits has same energy with orbits in $\widetilde {\mathcal{N}}_{v_+}, \widetilde {\mathcal{N}}_{u_-}$, thus are contained in $H^{-1}(0)$. However, due to the identity $dH(X_H)=-dH(\cR)H$, the energy conservation law fails for general contact Hamiltonian flows. Accordingly, in Section 6.3, we show, through a concrete example, that the connecting orbits can have non-zero energy.
\end{remark}

\subsection{Outline of the paper}
The rest of this paper is organized as follows. In Section 2, we recall necessary definitions, in particular action functions, and results from \cite{WWY1}-\cite{WWY3} and use them to establish a global version of characteristic method for solving \eqref{HJe}. In Section 3, we prove Theorem A by employing theorems developed in the last section. Section 4 is devoted to the proof of Theorem B. In Section 5, we shall define the notions of action minimizing orbits as well as time-free minimizers for contact Hamiltonian systems. As an application of our main results, we analyse the toy model Example \ref{toy-model} and give a complete classification of the asymptotic behavior of global action minimizing orbits in the non-critical case. We also show, using concrete models, the connecting orbits constructed in (B2) can have non-zero energy. The last Appendix includes a list of properties of action functions and solution semigroups used in the main context and a deduction of decomposition theorem for semi-static orbits in Section 5.2.

\section{Global characteristics method}
In this section, we extend the characteristic theory for evolutionary Hamilton-Jacobi equation \eqref{HJe} to the case when the time variable $t$ is not small. It is well known \cite[Chapter 1]{Ar-PDE} that characteristic theory addressed the local solvability of \eqref{HJe} assuming the initial data is smooth. For larger time, the solution to \eqref{HJe} becomes non-smooth. This attributes to the fact that the characteristics initiating from the $1$-graph of the initial data, after being projected by $\rho$ and suspended by $t$, intersect each other on $M\times[0,+\infty)$. However, we can apply variational method to pick up certain characteristics that are useful in building the $1$-pseudograph of the viscosity solution. Thus, we first recall

\subsection{The variational representation of solution semigroup of \eqref{HJe}}
The variational principle for contact Hamiltonian systems satisfying (H1)-(H3) was developed in \cite{WWY1}-\cite{WWY3}. Most of the definitions and propositions are listed without proof, we refer to the original articles for details.

\vspace{0.5em}
Let $TM$ be the tangent bundle of $M$. A point of $TM$ will be denoted by $(x,\dot{x})$, where $x\in M$ and $\dot{x}\in T_x M$. With slightly abuse of notations, use $\|\cdot\|_x$ to denote the norm induced by $g$ on $T_xM$ and $T^\ast_xM$. For a contact Hamiltonian $H$ satisfies (H1)-(H2), its convex dual $L:TM\times\R\rightarrow\R$,
\[
L(x,u,\dot{x})=\sup_{p \in T_x^*M}\{p \cdot\dot{x}-H(x,u,p)\}
\]
shares analogous assumptions \textbf{(H1)-(H3)} as $H$, i.e., fiberwise convexity and superlinearity in $v$ and uniformly Lipschitz in $u$.

\vspace{1em}
The following functions, though defined in an implicit way, are counterparts of the action functions in Aubry-Mather theory, see \cite[page 1364]{M2} and \cite[page 83]{S} for details. They contain all information about the variational principle defined by the \eqref{ch}. For another way to formulate the variational principle for contact Hamiltonian systems, which goes back to the work of G.Herglotz, we refer to \cite{CCWY},\cite{CCJWY}.

\begin{proposition}\cite[Theorem 2.1, 2.2]{WWY3}\label{Implicit variational}
For any given $x_0\in M$ and $u_0\in \R$, there exist continuous functions $h_{x_0,u_0}(x,t), h^{x_0,u_0}(x,t)$ defined on $M\times (0,+\infty)$ by
\begin{equation}\label{eq:Implicit variational}
\begin{split}
h_{x_0,u_0}(x,t)&=\inf_{\substack{\gamma(t)=x\\ \gamma(0)=x_0 } }\Big\{u_0+\int_0^t L(\gamma(s),h_{x_0,u_0}(\gamma(s),s),\dot \gamma(s))\ ds\Big\},\\
h^{x_0,u_0}(x,t)&=\sup_{\substack{\gamma(t)=x_0\\ \gamma(0)=x } }\Big\{u_0-\int_0^t L(\gamma(s),h^{x_0,u_0}(\gamma(s) ,t-s),\dot \gamma(s))\ ds\Big\},
\end{split}
\end{equation}
where the infimum and supremum are taken among Lipschitz continuous curves $\gamma:[0,t]\rightarrow M$ and are achieved. We call $h_{x_0,u_0}(x,t)$ the backward action function and $h^{x_0,u_0}(x,t)$ the forward action function.
\end{proposition}

\begin{remark}\label{minimizer-orbit}
Let $\gamma\in Lip([0,t],M)$ be the minimizer (resp. maximizer) of \eqref{eq:Implicit variational} and
\[
x(s):=\gamma(s), \quad u(s):=h_{x_0,u_0}(\gm(s),s)\,(\text{resp.}\,\,h^{x_0,u_0}(\gm(s),t-s)), \quad p(s):=\frac{\partial L}{\partial \dot x}(\gm(s),u(s),\dot\gm(s)).
\]
Then $(x(s),u(s),p(s))$ satisfies \eqref{ch} with $x(0)=x_0,x(t)=x\,($resp. $x(0)=x,x(t)=x_0)$ and
\[
\lim_{s\to 0^+ }u(s)=u_0\,(\text{resp. }\lim_{s\to t^-}u(s)=u_0).
\]
\end{remark}

As a direct consequence of Proposition \ref{Implicit variational} and Remark \ref{minimizer-orbit}, we obtain an equivalent definition of action functions. They are quite useful for obtaining information on general $\Phi^t_H$-orbits.
\begin{corollary}\label{Minimality}
Given $x_0,x\in M, u_0\in\R$ and $t>0$, set $(x(\tau),u(\tau),p(\tau))=\Phi^\tau_H(x(0),u(0),p(0))$ and
\[
S^{x,t}_{x_0,u_0}=\big\{(x(\tau),u(\tau),p(\tau)), \tau\in [0,t]\,:\, x(0)=x_0,x(t)=x,u(0)=u_0\big\},
\]
\[
S^{x_0,u_0}_{x,t}=\big\{(x(\tau),u(\tau),p(\tau)), \tau\in [0,t]\,:\, x(0)=q, x(t)=x_0, u(t)=u_0\big\},
\]
then for any $ (x,t)\in M\times(0,+\infty)$,
\begin{equation}\label{eq:inf}
h_{x_0,u_0}(x,t)=\inf\,\{u(t):(x(\tau),p(\tau),u(\tau))\in S^{x,t}_{x_0,u_0}\},
\end{equation}
\begin{equation}
h^{x_0,u_0}(x,t)=\sup \{u(0):(x(\tau),p(\tau),u(\tau))\in S^{x_0,u_0}_{x,t}\}.
\end{equation}
\end{corollary}

Now one can use action functions to give a variational formula for the solution semigroup of \eqref{HJe}. This is an extension of Lax-Oleinik operators \cite[Section 4.4]{Fathi_book} to contact Hamiltonian systems.

\begin{proposition}\cite[Proposition 4.1]{WWY2}\label{semi-group1}
For each $\varphi\in C(M,\R)$ and $(x,t)\in M\times(0,+\infty)$,
\begin{equation}\label{eq:Tt-+ rep}
\begin{split}
T^{-}_t\varphi(x)=\inf_{x_0\in M}h_{x_0,\varphi(x_0)}(x,t),\\
T^{+}_t\varphi(x)=\sup_{x_0\in M}h^{x_0,\varphi(x_0)}(x,t).
\end{split}
\end{equation}
In addition, we set $T^{\pm}_0\varphi(x)=\varphi(x)$, then for $t\geq0,\,\,T^{\pm}_t:\varphi\mapsto T^{\pm}_t\varphi$ maps $C(M,\R)$ to itself and satisfies all conditions listed in Section 1.3.
\end{proposition}

\begin{remark}\label{chara-mini}
Owing to the above definition, we notice that the viscosity solution to \eqref{HJe} propagates, i.e., can be calculate from integrating the Lagrangian, along minimizers defined by the \eqref{eq:Implicit variational}. This extends the notion of local characteristics of \eqref{HJe}, which are general orbits of $\Phi^t_H$ starting from 1-graph of the initial data. The essential use of minimizers comes from the property that the $\rho$-projection of two such minimizers do not intersect at intermediate points.
\end{remark}

\subsection{Characteristics for the solution semigroups}
The classical method of characteristics relates the \textbf{local solvability} of the Cauchy problem \eqref{HJe} to the study of orbits of \eqref{ch} near $\mathrm{J}^1_{\varphi}$, here $\varphi\in C^2(M)$ is a smooth initial data. More precisely, if $U:M\times[0,\delta]\rightarrow\R$ is a $C^2$ solution to \eqref{HJe}, then every trajectory $(x(t),u(t),p(t)), t\in[0,\delta]$ in $J^1(M,\R)$  satisfying the identities
\begin{equation}\label{ch-sol}
u(t)=U(x(t),t),\quad p(t)=\partial_x U(x(t),t).
\end{equation}
is an orbit segment $\Phi^t_H(Z)$ with some $Z\in\mathrm{J}^1_{\varphi}$. Conversely, for any $Z\in\mathrm{J}^1_{\varphi}$, the orbit segment $\Phi^t_H(Z)$ satisfies \eqref{ch-sol} on $t\in[0,\delta]$. In this sense, we call the orbit segment $\Phi^t_H(Z)=(x(t),u(t),p(t)), Z\in\mathrm{J}^1_{\varphi}, t\in[0,\delta]$ a \textbf{characteristic} for \eqref{HJe}. Furthermore, one has
\begin{theorem}\cite[Lecture 2, Theorem 3]{Ar-PDE} or \cite[Chapter 3, Theorem 2]{Evans-PDE}\label{local-chm}
For any $\varphi\in C^2(M,\R)$, there are $\delta>0$ and a solution $U\in C^2(M\times[0,\delta],\R)$ to \eqref{HJe} such that for any $t\in[0,\delta]$,
\begin{enumerate}[(1)]
  \item $\rho\circ\Phi^t_H:\mathrm{J}^1_{\varphi}\rightarrow M$ is a diffeomorphism,
  \item $\Phi^t_H(\mathrm{J}^1_{\varphi})=\mathrm{J}^1_{U(\cdot,t)}$.
\end{enumerate}
\end{theorem}

\vspace{0.5em}
From $(1)$, one notice that any two orbits do not collide within the time interval $[0,\delta]$. But this is not true for larger $t$, and it leads to non-differentiability of $U(\cdot,t)$. A global version of characteristics method for first order PDEs is obtained by extending the notion of $1$-graph of $U(\cdot,t)$ and characteristics for the solution semigroup, or viscosity solution, to \eqref{HJe}. The first is done in Definition \ref{pseudo-gr}. As is indicated in Remark \ref{chara-mini}, we shall show minimizers determined by action functions which, by Remark \ref{minimizer-orbit}, lifts to orbits of \eqref{ch}, are suitable candidates of characteristics.

\vspace{1em}
Recall that for an initial data $\varphi\in C(M,\R)$, we define $U:M\times[0,+\infty)\rightarrow\R$ by
\begin{equation}\label{sol-sg}
U(x,t):=T^{-}_t \varphi(x)=\inf_{x'\in M}h_{x',\varphi(x')}(x,t).
\end{equation}
By Proposition \ref{fundamental-prop} (3), fixing $(x,t)\in M\times(0,+\infty)$, the map
\[
x'\mapsto h_{x',\varphi(x')}(x,t)
\]
is continuous. Then there is a $x_0\in M$ such that $U(x,t)=h_{x_0,\varphi(x_0)}(x,t)$. Due to the properties of backward action function, we have
\begin{lemma}\label{sol-action}
For any minimizer $\gamma:[0,t]\rightarrow M$ of $h_{x_0,\varphi(x_0)}(x,t)$,
\[
U(\gamma(\tau),\tau)=h_{x_0,\varphi(x_0)}(\gamma(\tau),\tau).
\]
\end{lemma}

\begin{proof}
By \eqref{sol-sg}, we only need to show that for any $\tau\in[0,t]$,
\[
U(\gamma(\tau),\tau)\geq h_{x_0,\varphi(x_0)}(\gamma(\tau),\tau).
\]
We argue by contradiction. Assume there is $\tau\in(0,t)$ such that
\[
\underline{u}:=U(\gamma(\tau),\tau)<\bar{u}:=h_{x_0,\varphi(x_0)}(\gamma(\tau),\tau),
\]
then to complete the proof, it is necessary to see that
\[
U(x,t)=T^-_{t-\tau}U(\cdot,\tau)(x)\leqslant h_{\gamma(\tau),\underline{u}}(x,t-\tau)<h_{\gamma(\tau),\bar{u}}(x,t-\tau)=h_{x_0,\varphi(x_0)}(x,t)=U(x,t).
\]
Here, the first equality follows from the semigroup property and the second is a consequence of Proposition \ref{fundamental-prop} (2) and the fact that $\gamma$ is a minimizer of $h_{x_0,\varphi(x_0)}(x,t)$; the second inequality follows from Proposition \ref{fundamental-prop} (1).
\end{proof}

We need the fact that the $\rho$-projection of the local characteristic ensured by Theorem \ref{local-chm} is a minimizer in sense of Proposition \ref{Implicit variational}. Notice that by Theorem \ref{local-chm}, the map $\rho\circ\Phi^{\delta}_H:\mathrm{J}^1_{\varphi}\rightarrow M$ is a diffeomorphism, we use $(\rho\circ\Phi^{\delta}_H)^{-1}:M\rightarrow\mathrm{J}^1_{\varphi}$ to denote its inverse.

\begin{lemma}\label{local-mini}
For any $x\in M$ and $Z_0=(x_0,\varphi(x_0),d_x\varphi(x_0))=(\rho\circ\Phi^{\delta}_H)^{-1}(x)$, set
\[
\Phi^\tau_H(Z_0)=(x(\tau),u(\tau),p(\tau))\quad\text{for}\,\,\tau\in[0,\delta],
\]
then for all $\tau\in[0,\delta], u(\tau)=h_{x_0,\varphi(x_0)}(x(\tau),\tau)$ and
\[
h_{x_0,\varphi(x_0)}(x,\delta)=\varphi(x_0)+\int_0^\delta L(x(\tau),h_{x_0,\varphi(x_0)}(x(\tau) ,\tau),\dot{x}(\tau))\ d\tau.
\]
\end{lemma}

\begin{proof}
By Theorem \ref{local-chm}, $U\in C^2(M\times[0,\delta],\R)$ and $(x(\tau),p(\tau),u(\tau))$ satisfy
\begin{equation}\label{eq:2-1}
u(\tau)=U(x(\tau),\tau),\quad p(\tau)=\partial_x U(x(\tau),\tau),\quad\tau\in[0,\delta].
\end{equation}
and the boundary conditions read as
\begin{equation}\label{eq:2-2}
x(0)=x_0,\quad x(\delta)=x,\quad u(0)=\varphi(x_0).
\end{equation}
It follows from Proposition \ref{semi-group1} that
\[
h_{x_0,\varphi(x_0)}(x(\tau),\tau)\geq T^-_\tau\varphi(x(\tau))=U(x(\tau),\tau)=u(\tau).
\]
Combining \eqref{eq:2-2} and Corollary \ref{Minimality} gives
\[
h_{x_0,\varphi(x_0)}(x(\tau),\tau)=\inf\,\{u(\tau):(x(t),p(t),u(t))\in S^{x(\tau),\tau}_{x_0,\varphi(x_0)}\}\leq u(\tau)
\]
and therefore for $\tau\in[0,\delta]$,
\[
u(\tau)=h_{x_0,\varphi(x_0)}(x(\tau) ,\tau).
\]
Now we can compute as
\begin{align*}
&h_{x_0,\varphi(x_0)}(x,\delta)\geq U(x,\delta)=U(x(\delta),\delta)\\
=\,&U(x(0),0)+\int^\delta_0 \partial_t U(x(\tau),\tau)+\langle\partial_x U(x(\tau),\tau),\dot{x}(\tau)\rangle\,d\tau\\
=\,&\varphi(x_0)+\int^\delta_0 \partial_t U(x(\tau),\tau)+\langle p(\tau),\partial_p H(x(\tau),u(\tau),p(\tau))\rangle\,d\tau\\
=\,&\varphi(x_0)+\int^\delta_0 \partial_t U(x(\tau),\tau)+H(x(\tau),u(\tau),p(\tau))+L(x(\tau),u(\tau),\dot{x}(\tau))\,d\tau\\
=\,&\varphi(x_0)+\int^\delta_0 L(x(\tau),u(\tau),\dot{x}(\tau))\,d\tau\\
=\,&\varphi(x_0)+\int^\delta_0 L(x(\tau),h_{x_0,\varphi(x_0)}(x(\tau) ,\tau),\dot{x}(\tau))\,d\tau.
\end{align*}
Here, the third equality uses \eqref{eq:2-1} and the equations \eqref{ch} for characteristics, the fourth equality follows from the knowledge of Legendre-Fenchel inequality in convex analysis, i.e.,
\[
\dot{x}=\partial_p H(x,u,p)\Leftrightarrow\langle p,\dot{x}\rangle=H(x,u,p)+L(x,u,\dot{x}),
\]
the fifth equality is due to \eqref{eq:2-1} and the fact that $U(x,t)$ is a solution to \eqref{HJe}. Combining the above inequality and \eqref{eq:Implicit variational}, we complete the proof.
\end{proof}

Now we could verify the first equation in \eqref{ch-sol} by showing the conclusion of the above lemma holds for arbitrary $t>0$. In fact we have
\begin{theorem}\label{global-chm1}
Assume $\varphi\in C^2(M,\R)$. For any $(x,t)\in M\times(0,+\infty)$, there is $Z_0\in\mathrm{J}^1_\varphi$ such that the characteristic segment $\Phi^\tau_H(Z_0)=(x(\tau),u(\tau),p(\tau))$ satisfies
\begin{equation}\label{ch-me-u}
u(\tau)=h_{x_0,\varphi(x_0)}(\gamma(\tau),\tau)=U(x(\tau),\tau)\quad\quad\text{for}\quad\tau\in[0,t].
\end{equation}
\end{theorem}

\begin{proof}
For $t\leqslant \delta$, Theorem \ref{global-chm1} reduces to Theorem \ref{local-chm} and there is nothing to prove. For $t>\delta$, we use \eqref{sol-sg} to write
\[
U(x,t)=T^-_t \varphi(x)=T^-_{t-\delta}\circ T^-_{\delta}\varphi(x)=\inf_{x'\in M}h_{x',T^-_{\delta}\varphi(x')}(x,t-\delta).
\]
Proposition \ref{fundamental-prop} (4) and \ref{prop-sg} (3) imply $h_{x',T^-_{\delta}\varphi(x')}(x,t)$ is Lipschitz continuous in $x'$. Since $M$ is compact, the above infimum is attained at $x'=x_1\in M$. Set $u_1=T^-_{\delta}\varphi(x_1)$, then according to Proposition \ref{Implicit variational}, there is a minimizer $\gamma_1:[0,t-\delta]\rightarrow M$ with $\gamma_1(0)=x_1$ and
\begin{equation}\label{min-1}
U(x,t)=h_{x_1,u_1}(x,t-\delta)=u_1+\int_0^{t-\delta} L(\gamma_1(\tau), h_{x_1,u_1}(\gamma_1(\tau) ,\tau),\dot{\gamma_1}(\tau))\ d\tau.
\end{equation}
For $\tau\in[\delta,t]$, we set $x_1(\tau)=\gamma_1(\tau-\delta)$ and
\[
u_1(\tau):=h_{x_1,u_1}(\gamma_1(\tau-\delta),\tau-\delta),\quad p_1(\tau):=\frac{\partial L}{\partial \dot x}(\gamma_1(\tau-\delta),u_1(\tau),\dot{\gamma}_1(\tau-\delta)),
\]
then Proposition \ref{Implicit variational} also implies that $(x_1(\tau),u_1(\tau),p_1(\tau))$ satisfies \eqref{ch} and
\begin{align*}
x_1(\delta)=x_1,\quad x_1(t)=x,\quad\lim_{\tau\rightarrow t^-}u_1(\tau)=u_1.
\end{align*}

By Lemma \ref{local-mini}, there is a unique $Z_0=(x_0, \varphi(x_0), d_x \varphi(x_0))\in\mathrm{J}^1_\varphi$ such that
\[
\rho\circ\Phi^{\delta}_H(z _0)=x_1,
\]
and set $\Phi^\tau_H(Z_0)=(x_0(\tau),u_0(\tau),p_0(\tau)),\tau\in[0,\delta]$, then
\begin{equation}\label{min-2}
h_{x_0,\varphi(x_0)}(x_1,\delta)=\varphi(x_0)+\int_0^\delta L(x_0(\tau),h_{x_0,\varphi(x_0)}(x_0(\tau) ,\tau), \dot{x}_0(\tau) )\ d\tau
\end{equation}
and
\[
h_{x_0,\varphi(x_0)}(x_1,\delta)=u_0(\delta)=U(x_1,\delta)=T^-_{\delta}\varphi(x_1)=u_1.
\]

\vspace{0.5em}
\textbf{Claim:}\,\,for $\tau\in[\delta,t]$,
\begin{equation}\label{eq:3}
h_{x_0,\varphi(x_0)}(x_1(\tau),\tau)=h_{x_1,u_1}(x_1(\tau),\tau-\delta).
\end{equation}

\vspace{0.5em}
\textit{Proof of the claim}:\,\,It follows from Proposition \ref{fundamental-prop} (2) that
\[
h_{x_0,\varphi(x_0)}(x_1(\tau),\tau)=\inf_{x'\in M}h_{x',h_{x_0,\varphi(x_0)}(x',\delta)}(x_1(\tau),\tau-\delta)\leq h_{x_1,u_1}(x_1(\tau),\tau-\delta).
\]
Assume for some $\tau\in(\delta,t)$,
\[
u_2:=h_{x_0,\varphi(x_0)}(x_1(\tau),\tau)<h_{x_1,u_1}(x_1(\tau),\tau-\delta):=\bar{u}_2,
\]
then by Proposition \ref{fundamental-prop} (1) and the fact that $\gamma_1$ is a minimizer of $h_{x_1,u_1}(x,t-\delta)$,
\[
h_{x_0,\varphi(x_0)}(x,t)\leqslant h_{x_1(\tau),u_2}(x,t-\tau)<h_{x_1(\tau),\bar{u}_2}(x,t-\tau)=h_{x_1,u_1}(x,t-\delta)=U(x,t).
\]
This contradicts to Proposition \ref{semi-group1} since
\[
U(x,t)=T^-_t \varphi(x)=\inf_{x'\in M}h_{x',\varphi(x')}(x,t)\leqslant h_{x_0,\varphi(x_0)}(x,t).
\]
\qed

Define $x:[0,t]\rightarrow M$ by
\[
x(\tau):=
\begin{cases}
x_0(\tau),\quad \tau\in[0,\delta];\\
x_1(\tau),\quad \tau\in[\delta,t],
\end{cases}
\]
then combining \eqref{min-1}-\eqref{eq:3}, we obtain
\begin{align*}
&\,h_{x_0,\varphi(x_0)}(x,t)=h_{x_1,u_1}(x,t-\delta)=u_1+\int_0^{t-\delta} L(\gamma_1(\tau),h_{x_1,u_1}(\gamma_1(\tau) ,\tau), \dot{\gamma_1}(\tau))\ d\tau\\
=&\,\,u_1+\int_{\delta}^{t} L(x_1(\tau),h_{x_1,u_1}(x_1(\tau),\tau-\delta), \dot{x_1}(\tau))\ d\tau\\
=&\,\,h_{x_0,\varphi(x_0)}(x_1,\delta)+\int_{\delta}^{t} L(x_1(\tau),h_{x_0,\varphi(x_0)}(x_1(\tau),\tau), \dot{x_1}(\tau))\ d\tau\\
=&\,\,\varphi(x_0)+\int_0^\delta L(x_0(\tau),h_{x_0,\varphi(x_0)}(x_0(\tau) ,\tau), \dot{q}_0(\tau) )\ d\tau+\int_{\delta}^{t} L(x_1(\tau),h_{x_0,\varphi(x_0)}(x_1(\tau),\tau), \dot{x_1}(\tau))\ d\tau\\
=&\,\,\varphi(x_0)+\int_0^t L(x(\tau),h_{x_0,\varphi(x_0)}(x(\tau),\tau), \dot{x}(\tau))\ d\tau,\\
\end{align*}
This shows that $x:[0,t]\rightarrow M$ is a minimizer of $h_{x_0,\varphi(x_0)}(x,t)$, thus by Proposition \ref{Implicit variational},
\[
u(\tau)=h_{x_0,\varphi(x_0)}(\gamma(\tau),\tau),\quad p(\tau)=\frac{\partial L}{\partial \dot x}(x(\tau),u(\tau),\dot{x}(\tau))
\]
is a $C^1$ characteristic starting from $z _0$. Invoking Lemma \ref{sol-action}, we find
\[
u(\tau)=
\begin{cases}
u_0(\tau)=U(x_0(\tau),\tau)=U(x(\tau),\tau),\quad &\tau\in[0,\delta];\\
h_{x_1,u_1}(x_1(\tau),\tau-\delta)=U(x_1(\tau),\tau)=U(x(\tau),\tau),\quad &\tau\in[\delta,t].
\end{cases}
\]
\end{proof}

Moreover, it is well known in the theory of one dimensional calculus of variation that
\begin{lemma}\label{global-chm2}
Assume $\varphi\in\text{Lip}(M,\R)$. With the same construction as above, $U$ is differentiable at $(x(\tau),\tau)$ for $\tau\in(0,t)$ and $p(\tau)=\partial_x U(x(\tau),\tau)$ for $\tau\in [0,t)$. Thus $p(t)\in D^{\ast}_x U(x(t),t)$.
\end{lemma}

\begin{proof}
Since $U:M\times[0,+\infty)\rightarrow\R$ is locally Lipschitz, \cite[Theorem 5.3.8]{CS} implies that $U$ is locally semi-concave. Using this fact along with a similar argument as \cite[Theorem 6.4.7]{CS}, $U$ is differentiable at $(x(\tau),\tau)$ for $\tau\in(0,t)$. And we have
\begin{align*}
\dot u(\tau)=&\, \frac{d}{d\tau} U(x(\tau),\tau)
= \partial_t U(x(\tau),\tau) +\partial_x U(x(\tau),\tau) \cdot \dot x(\tau) \\
=&\, -H\Big (x(\tau),   U(x(\tau),\tau) , \partial_x U(x(\tau),\tau) \Big) + \partial_x U(x(\tau),\tau) \cdot \dot x(\tau) \\
=&\, -H\Big (x(\tau),   u(\tau) , \partial_x U(x(\tau),\tau) \Big) + \partial_x U(x(\tau),\tau) \cdot  H_p(x(\tau),p(\tau),u(\tau) )
\end{align*}
where the third equality owes to the fact that $U(x,t) $ is a solution to \eqref{HJe}. Setting $P(\tau)=\partial_x U(x(\tau),\tau)$, we have, according to the above computation and \eqref{ch},
$$
0=\dot u(\tau)-\dot u(\tau)=H\Big (x(\tau),u(\tau),P(\tau)\Big)-H\Big(x(\tau),u(\tau),p(\tau)\Big )- H_p(x(\tau),u(\tau),p(\tau))(P(\tau)-p(\tau))\geq0,
$$
the last inequality uses the strict convexity assumption (H1) of $H$. Thus $p(\tau)=P(\tau)=\partial_x U(x(\tau),\tau)$. The last conclusion follows from $p(t)=\lim_{\tau\rightarrow t}p(\tau)=\lim_{\tau\rightarrow t}\partial_x U(x(\tau),\tau)\in D_x^{\ast}U(x(t),t)$.
\end{proof}


We summarize Theorem \ref{global-chm1} and Corollary \ref{global-chm2} into the following
\begin{theorem}(Global characteristics method)\label{global-chm}
Assume $\varphi\in C^2(M,\R)$. For any $(x,t)\in M\times(0,+\infty)$, there is $Z_0\in\mathrm{J}^1_\varphi$ such that the orbit segment $\Phi^\tau_H(Z_0)=(x(\tau),p(\tau),u(\tau))$ of \eqref{ch} satisfies
\begin{align*}\label{ch-me-u}
u(\tau)&=U(x(\tau),\tau)\quad\quad\text{for}\quad\tau\in[0,t],\\
p(\tau)&=\partial_x U(x(\tau),\tau)\quad\text{for}\quad\tau\in [0,t)\quad\text{and}\quad p(t)\in D^{\ast}_x U(x(t),t).
\end{align*}
\end{theorem}

\begin{remark}
The theorem shows that even for large $t$, the solution of \eqref{HJe} can also be traced by $\Phi^\tau_H$-orbits starting from the 1-graph of the initial data. The main difference from the local case is that the map $\rho\circ\Phi^t_H:\mathrm{J}^1_\varphi\rightarrow M$ is only a \textbf{surjection} rather than a diffeomorphism. We shall call the $\Phi^\tau_H$-orbits found by Theorem \ref{global-chm} the \textbf{characteristics for the solution semigroup}.
\end{remark}

\section{Proof of Theorem \ref{thm1}}
This section is devoted to a proof of our first main result. In this section, we always assume that $\varphi\in C^2(M,\R)$ and $\lim_{t\rightarrow+\infty}T^-_t\varphi=u_-$ uniformly for $x\in M$. The first conclusion of the theorem is an abstract mechanism for producing \textbf{semi-infinite} orbits connecting a Legendrian graph $\mathrm{J}^1_{\varphi}$ and a Legendrian pseudograph $\mathrm{J}^1_{u_-}$. The second conclusion shows the convergence of solution semigroup is actually $C^1$. Before giving the proof, we recall

\subsection{Calibrated properties of weak KAM solutions}
Calibrated properties reveal some duality between functions or differential forms defined by variational principle and geometric objects, e.g. curves, surfaces, attaining the extremum of the corresponding variational problem. For weak KAM solutions, we have
\begin{proposition}\label{weak-kam}\cite[Proposition 2.7]{WWY3}
$\,\,u\in\mathcal{S}_-\,\,($resp. $u_+\in\mathcal{S}_+)$ if and only if
\begin{enumerate}[(1)]
  \item $u_\pm$ is dominated by $L$,  i.e., for any continuous piecewise $C^1$ curve $\gamma:[a,b]\rightarrow M$,
        \begin{equation}\label{L-dom}
        u_\pm(\gamma(b))-u_\pm(\gamma(a))\leqslant\int^{b}_a L(\gamma(s),u_\pm(\gamma(s)),\dot{\gm}(s))\ ds.
        \end{equation}
        In general, the notation $u\prec L$ means $u\in C(M,\R)$ is dominated by $L$.

  \item for any $x\in M$, there is a $C^1$ curve $\gamma:(-\infty,0]$ (resp. $[0,+\infty))\rightarrow M$ with $\gamma(0)=x$ and
        \begin{equation}\label{L-dom}
        \begin{split}
        &u_-(x)-u_-(\gamma(t))=\int^{0}_t L(\gamma(s),u_-(\gamma(s)),\dot{\gm}(s))\ ds,\quad\forall t<0;\\
        (\text{resp.}\quad &u_+(\gamma(t))-u_+(x)=\int^{t}_0 L(\gamma(s),u_+(\gamma(s)),\dot{\gm}(s))\ ds, \quad\forall t>0.)
        \end{split}
        \end{equation}
        If $u_\pm\prec L$, the curves defined above are called $(u_-,L)\,\,($resp.$(u_+,L))$-\textbf{calibrated curves}.
\end{enumerate}
\end{proposition}

From the above proposition, one deduce the regularity result
\begin{corollary}\cite[Lemma 4.1]{WWY3}\label{weak-kam-reg}
If $u\prec L$, then $u$ is Lipschitz continuous on $M$. Moreover, all backward (resp. forward) weak KAM solutions are locally semiconcave (resp. semiconvex) on $M$.
\end{corollary}

\begin{proof}
The first conclusion is the content of \cite[Lemma 4.1]{WWY3}. The second conclusion follows from the first conclusion, Proposition \ref{weak-kam-vis} and a direct application of \cite[Theorem 5.3.7]{CS}.
\end{proof}

The next propositions describe the geometric structure of the Legendrian pseudo-graph of backward/forward weak KAM solutions. The first is a summary of results obtained in \cite{WWY3}. Although the paper \cite{WWY3} focus on monotone systems, the results presented here is proved under the general assumption (H1)-(H3).
\begin{proposition}\cite[Lemma 4.1-4.3]{WWY3}\label{cali-pro}
Let $u_-\in\mathcal{S}_-\,\,($resp. $u_+\in\mathcal{S}_+)$ and $\gm:(-\infty,0]\,\,($resp. $[0,+\infty))\rightarrow M$ be a $(u_-,L)\,\,($resp.$(u_+,L))$-calibrated curve. Setting
\[
z(s)=(\gm(s),u_\pm(\gm(s)),p(s)),\quad p(s)=\frac{\partial L}{\partial \dot x}(\gm(s),u_\pm(\gm(s)),\dot{\gm}(s)),
\]
then
\begin{enumerate}[(1)]
  \item $u_\pm$ is differentiable at $\gm(s)$ if $s\neq0$,
  \item $d_x u_\pm(\gm(s))=p(s)$ and $H(z(s))=0$ if $u_\pm$ is differentiable at $\gm(s)$,
  \item $\Phi^t_H(z(s))=z(t+s)$ for $s, t\in(-\infty,0)\,\,($resp. $(0,+\infty))$.
\end{enumerate}
\end{proposition}

Furthermore, even if the weak KAM solution is non-differentiable at some point $x\in M$, we have
\begin{proposition}
With the same assumption as Proposition \ref{cali-pro}, there is a one to one correspondence between reachable gradients of $u$ at $x\in M$ and $(u,L)$-calibrated curves based at $x:$
\begin{enumerate}[(1)]
  \item for any $(u_-,L)\,\,($resp.$(u_+,L))$-calibrated curve $\gm:(-\infty,0]\,\,($resp. $[0,+\infty))\rightarrow M$ with $\gm(0)=x$,
      \[
      p(0):=\frac{\partial L}{\partial \dot x}(x,u_-(x),\dot{\gm}_-(0))\in D^\ast u_-(x)\,\,(\text{resp. }\frac{\partial L}{\partial \dot x}(x,u_+(x),\dot{\gm}_+(0))\in D^\ast u_+(x)).
      \]
  \item for any $p\in D^\ast u_-(x)\,\,($resp. $p\in D^\ast u_+(x))$, there is a \textbf{unique} $(u_-,L)\,\,($resp.$(u_+,L))$-calibrated curve $\gm:(-\infty,0]\,\,($resp. $[0,+\infty))\rightarrow M$ with $\gm(0)=x$ such that
      \[
      p=\frac{\partial L}{\partial \dot x}(x,u_-(x),\dot{\gm}_-(0))\,\,(\text{resp. }\frac{\partial L}{\partial \dot x}(x,u_+(x),\dot{\gm}_+(0))).
      \]
\end{enumerate}
Here $\dot{\gm}_\pm(0)$ denote the one-sided derivative of $\gm$ at $x$.
\end{proposition}

\begin{proof}
We shall focus on the case $u_-\in\mathcal{S}_-$, the other one is proved in a similar fashion. We apply Proposition \ref{cali-pro} to conclude that $z|_{(-\infty,0)}$ is bounded orbit of $\Phi^t_H$. Thus $\gm\in C^1$ and we have
\begin{equation}\label{left-continuous}
\lim_{s\rightarrow0^-}\dot{\gm}(s)=\dot{\gm}_-(0),
\end{equation}
where $\dot{\gm}_-(0)$ denotes the left derivative of $\gm$ at $0$. Now we define
\[
x(0)=\gm(0),\quad p(0)=\frac{\partial L}{\partial \dot x}(\gm(0),u(\gm(0)),\dot{\gm}_-(0)),\quad u(0)=u(\gm(0)).
\]
By continuity of $\Phi^t_H$ and $z$, it is not hard to verify that $\phi^t_H(z(0))=z(t)$ for any $t\in(-\infty,0)$.

\vspace{1em}
\textbf{(1)}: Notice that by Proposition \ref{cali-pro}, $u_-$ is differentiable at $\gm(s), s<0$ and
\[
d_x u_-(\gm(s))=\frac{\partial L}{\partial\dot{x}}(\gm(s),u_-(\gm(s)),\dot{\gm}(s)).
\]
It follows from \eqref{left-continuous} that
\[
p(0)=\frac{\partial L}{\partial\dot{x}}(x,u_-(x),\dot{\gm}_-(0))=\lim_{s\rightarrow0^-}\frac{\partial L}{\partial\dot{x}}(\gm(s),u_-(\gm(s)),\dot{\gm}(s))=\lim_{s\rightarrow0^-}d_x u_-(\gm(s))\in D^\ast u_-(x).
\]

\vspace{1em}
\textbf{(2)}: For each $p\in D^\ast u_-(x)$, there is a sequence $\{x_n\}\subset M$ converges to $x$  such that $u_-$ is differentiable at each $x_n$ and the sequence $\displaystyle\lim_{n\rightarrow\infty}d_x u_-(x_n)=p$. By Proposition \ref{weak-kam} and \ref{cali-pro}, there is a $(u_-,L)$-calibrated curve $\gm_n:(-\infty,0]\rightarrow M$ with $\gm_n(0)=x_n$ and
\[
d_x u_-(x_n)=\frac{\partial L}{\partial\dot{x}}(x_n,u_-(x_n),\dot{\gm}_{n,-}(0)).
\]
It is easy to see that for $t\leqslant0, \{z_n(t)=(\gm_n(t),u_-\circ\gm_n(t),p_n(t)): p_n(t)=\frac{\partial L}{\partial\dot{x}}(\gm_n(t),u_-\circ\gm_n(t),\dot{\gm}_n(t))\}_{n\geqslant1}$ are uniformly bounded $\phi^t_H$-orbits. Thus $\gm_n, \dot{\gm}_n$ are equi-bounded and equi-Lipschitz curves. Up to subsequences, $\gm_n$ and $\dot{\gm}_n$ converges to $\gm$ and $\dot{\gm}:(-\infty,0]\rightarrow M$ uniformly on compact intervals. It follows that $\gm$ is also a $(u_-,L)$-calibrated curve and
$$
\frac{\partial L}{\partial\dot{x}}(\gm(0),u_-(\gm(0)),\dot{\gm}_-(0))=\lim_{n\rightarrow\infty}\frac{\partial L}{\partial\dot{x}}(\gm_n(0),u_-(\gm_n(0)),\dot{\gm}_{n,-}(0))=\lim_{n\rightarrow\infty}d_x u_-(x_n)=p.
$$
The uniqueness of such calibrated curve follows from the uniqueness of solution to \eqref{ch} with initial data $z(0)=(x,u_-(x),p)$.
\end{proof}

\begin{corollary}\label{inv-graph}
Assume $u_-\in\mathcal{S}_- \,($resp. $u_+\in\mathcal{S}_+)$. Then $\mathrm{J}^1_{u_-}\,($resp. $\mathrm{J}^1_{u_+})$ is backward (resp. forward) invariant, i.e., $\Phi^t_H(\mathrm{J}^1_{u_-})\subset\mathrm{J}^1_{u_-},\,\, t\in(-\infty,0]\quad(\text{resp. }\Phi^t_H(\mathrm{J}^1_{u_+})\subset\mathrm{J}^1_{u_+},\,\, t\in[0,+\infty))$. In particular, $\widetilde{\mathcal{N}}_{u_\pm}\subset\mathrm{J}^1_{u_\pm}$ and $\mathrm{J}^1_{u_-}=\overline{\cup_{\delta>0}\Phi_H^{-\delta}(\mathrm{J}^1_{u_-})}, \,\,($resp. $\mathrm{J}^1_{u_+}=\overline{\cup_{\delta>0}\Phi_H^{\delta}(\mathrm{J}^1_{u_+})})$.
\end{corollary}

\subsection{Proof of \textbf{(A1)} with first applications}
For $n\geqslant1$, we apply Theorem \ref{global-chm1} to choose any sequence of characteristic segments
\[
z _n(t):=\Phi^{t}_H(Z_n)=(x_n(t), u_n(t), p_n(t)),\quad t\in[0,\mathcal{T}_n]
\]
of $T^-_t\varphi$ with $Z_n=(x_n,   u_n,p_n)\in\mathrm{J}^1_{\varphi}$. Here, the only requirement is $\lim_{n\rightarrow+\infty}\mathcal{T}_n=+\infty$. Then, up to a subsequence,
\begin{equation}\label{cali}
\lim_{n\rightarrow\infty}Z_n=Z_0=(x_0, \varphi (x_0), d_x \varphi (x_0))\in\mathrm{J}^1_{\varphi},\quad u_n(t)=T^-_t \varphi (x_n(t)).
\end{equation}
Due to the uniform Lipschitz property of $\{T^-_t \varphi\}_{t\geq0}$, $\{z_n(t):t\geq0\}_{n\geqslant1}$ is uniform bounded. The continuous dependence theorem of ODE implies $z_n(t)$ converges to a semi-infinite orbit $z:[0,+\infty)\rightarrow J^1(M,\R); z(t)=(x(t),u(t),p(t))$ uniformly on compact intervals with
\begin{equation}\label{semi-infi-ch}
z(0)=Z_0\quad\text{and}\quad u(t)=T^-_t \varphi (x(t)), \quad \forall t\in[0,+\infty),
\end{equation}
where the second equality follows from \eqref{cali} and the continuity of the function $T^-_t \varphi$. This means that for any $t\geqslant0$, the restriction $z|_{[0,t]}$ is a characteristic for $T^-_t\varphi$.

\vspace{0.5em}
Notice that $\{z(t):t\geqslant0\}$ is contained in a compact subset of the phase space, so $\omega(Z_0)$ is nonempty. Moreover, from the very definition, $\omega(Z_0)$ is compact and $\Phi^t_H$-invariant set, that is, if $\bar{z}\in\omega(Z_0)$, then $\Phi^{t}_H(\bar{z})$ is defined for all $t\in\R$ and belongs to $\omega(Z_0)$. By Definition \ref{Mane-slice}, we reduce the conclusion to

\begin{proposition}\label{omega-mane}
Assume $\varphi\in C(M,\R)$ and $\lim_{t\rightarrow\infty}T^-_t\varphi=u_-$. Then for any semi-infinite orbit $\Phi^t_H(Z_0):=z(t)=(x(t),u(t),p(t)), t\geqslant0$ such that \eqref{semi-infi-ch} holds, $\omega(Z_0)\subset\mathrm{J}^1_{u_-}$.
\end{proposition}
For any $\bar{z}=(\bar{x},\bar{u},\bar{p})\in\omega(Z_0)$, there is non-negative sequence $\{t_j\}_{j\geqslant 1}$ with
\[
\lim_{j\rightarrow\infty}t_j=+\infty\text{ and }(\bar{x},\bar{u},\bar{p})=\displaystyle \lim_{j\rightarrow\infty}(x(t_j),u(t_j),p(t_j))= \displaystyle \lim_{j\rightarrow\infty}(x(t_j),T^-_{t_j}\varphi (x(t_j)) ,p(t_j) ).
\]
We prove the proposition by verifying two lemmas.

\vspace{0.5em}
\begin{lemma}\label{proof-A2-u}
Under the assumption of Proposition \ref{omega-mane} and the notation above, $\bar{u}=u_-(\bar{x})$.
\end{lemma}

\begin{proof}
For any $\eps>0$, there is $N_1\in\mathbb{N}$ such that for $j>N_1$,
\[
|\bar{u}-u(t_j)|<\eps,\quad |u_-(x(t_j))-u_-(\bar{x})|<\eps.
\]
On the other hand, \eqref{convergence1} implies that there is $N_2\in\mathbb{N}$ such that for $j>N_2$,
\[
|T^-_{t_j} \varphi (x)-u_-(x)|<\eps,\quad\text{for all}\quad x\in M.
\]
Thus for $j>\max\{N_1,N_2\}$,
\begin{align*}
|\bar{u}-u_-(\bar{x})|&\,\leq|\bar{u}-u(t_j)|+|u(t_j)-u_-(x(t_j))|+|u_-(x(t_j))-u_-(\bar{x})|\\
&\,\leq|\bar{u}-u(t_j)|+|T^-_{t_j}\varphi  (x(t_j))-u_-(x(t_j))|+|u_-(x(t_j))-u_-(\bar{x})|\\
&\,<3\eps.
\end{align*}
This completes the proof.
\end{proof}
By the definition of calibrated curve in Proposition \ref{weak-kam} (2), it is not difficult to prove the following lemma. Here we give a proof independent of the notion of calibrated curves.
\begin{lemma}\label{proof-A2-p}
Assume $u_-\in\mathcal{S}_-$ and for some $\delta>0$,
\[
\pi\,(\{\Phi^{t}_H(\bar{x},\bar{u},\bar{p}):t\in[-\delta,\delta]\})\subset J^0_{u_-},
\]
then $\bar{p}\in D^\ast_x u_-(\bar{x})$.
\end{lemma}
\begin{proof}
We assume $\bar{p}\notin D^{\ast}_x u_-(\bar{x})$ to conclude a contradiction. For any $p\in D^{\ast}_x u_-(\bar{x})$, we set
\begin{equation}
\begin{split}
\bar{z}(t)&=(\bar{x}(t),\bar{u}(t),\bar{p}(t))=\Phi^{t}_H(\bar{x},\bar{u},\bar{p}),\quad t\in[-\delta,\delta],\\
\hat{z}(t)&=(\hat{x}(t),\hat{u}(t),\hat{p}(t))=\Phi^{t}_H(\bar{x},\bar{u},p),\quad t\in(-\infty,0].
\end{split}
\end{equation}
By the assumption and backward invariance of $\mathrm{J}^1_{u_-}$, we have
\begin{equation}\label{ch-u}
\bar{u}(t)=u_-(\bar{x}(t))\quad\text{for}\,\,\,t\in[-\delta,\delta],\quad\hat{u}(t)=u_-(\hat{x}(t))\quad\text{for}\,\,\,t\in(-\infty,0].
\end{equation}
In particular, $\bar{x}(0)=\hat{x}(0)=\bar{x}, \bar{u}(0)=\hat{u}(0)=u_-(\bar{x})$. Since $u_-\in\mathcal{S}_-$, then by \eqref{ch-u},
\[
\bar{u}=u_-(\bar{x})=T^-_\delta u_-(\bar{x})\leqslant h_{\hat{x}(-\delta),u_-(\hat{x}(-\delta))}(\bar{x},\delta)=h_{\hat{x}(-\delta),\hat{u}(-\delta)}(\bar{x},\delta)\leqslant \bar{u},
\]
where the last inequality follows from Corollary \ref{Minimality}. The contradiction follows from
\begin{align*}
u_-(\bar{x}(\delta))&\,=T^-_2 u_-(\bar{x}(\delta))\leq h_{\hat{x}(-\delta),u_-(\hat{x}(-\delta))}(\bar{x}(\delta),2\delta)\\
&\,=h_{\hat{x}(-\delta),\hat{u}(-\delta)}(\bar{x}(\delta),2\delta)<h_{\bar{x},h_{\hat{x}(-\delta),\hat{u}(-\delta)}(\bar{x},\delta)}(\bar{x}(\delta),\delta)\\
&\,=h_{\bar{x},\bar{u}}(\bar{x}(\delta),\delta)\leq\bar{u}(\delta)=u_-(\bar{x}(\delta)),
\end{align*}
where the second inequality uses Markov property, i.e. Proposition \ref{fundamental-prop} (2), of action function and the fact that the minimizers $(\hat{x}(t),\hat{u}(t)),t\in[-\delta,0]$ of $h_{\hat{x}(-\delta),\hat{u}(-\delta)}(\bar{x},\delta)$ and $(\bar{x}(t),\bar{u}(t)), t\in[0,\delta]$ of $h_{\bar{x},\bar{u}}(\bar{x}(\delta),\delta)$ has a corner at $(\bar{x},\bar{u})$, thus can not be the minimizer of $h_{\hat{x}(-\delta),\hat{u}(-\delta)}(\bar{x}(\delta),2\delta)$. The last inequality is again a consequence of Corollary \ref{Minimality}.
\end{proof}

\begin{remark}\label{proof-A2-p'}
By a slight adaption of the proof, Lemma \ref{proof-A2-p} holds true with $u_-$ replaced by $u_+\in\mathcal{S}_+$.
\end{remark}

To complete the proof of Proposition \ref{omega-mane}, it suffices to notice that, by Lemma \ref{proof-A2-u}, $\omega(Z_0)$ is $\Phi^t_H$-invariant and $\pi(\omega(Z_0))$ is a subset of $J^0_{u_-}$, then apply Lemma \ref{proof-A2-p} directly.

\begin{remark}
The above mechanism has the advantages that it requires no information about the dynamics of $\Phi^t_H$ in a neighborhood of $\mathrm{J}^1_{u_-}$. However, for a fixed solution $u_-$, the convergence assumption \eqref{convergence1} may hold only for a small set of initial data.
\end{remark}

In response to the above remark, we present an example showing that the monotone assumption is not necessary for the uniqueness of solution of \eqref{hj} as well as the convergence of solution semigroup on the whole space of continuous functions.
\begin{proposition} \cite[Theorem 2.1]{XYZ} \label{prop:XYZ}
Assume $dH(\mathcal{R})>0$ on $E=\{z\in J^1(M,\R):H(z)=0 \}$ and the equation \eqref{hj} admits at least one solution. Then the solution to \eqref{hj} is unique, denoted by $u_-$, and for any $\varphi\in C(M,\R)$,
$$
\lim_{t\to +\infty} T_t^- \varphi(x) =u_-(x) , \quad x\in M.
$$
\end{proposition}

We combine (A1) and the above proposition to obtain
\begin{corollary}
Under the same assumption and notation as Proposition \ref{prop:XYZ}, for any $\varphi\in C^2(M,\R) $, there exists some $Z_\varphi\in \mathrm{J}^1_\varphi$ such that $\omega(Z_\varphi)$ is a nonempty subset of $\widetilde{\mathcal{N}}_{u_-}$.
\end{corollary}

\subsection{Proof of \textbf{(A2)} and some improvements}
Due to Corollary \ref{inv-graph}, the second conclusion of Theorem \ref{thm1} is implied by
\begin{proposition}\label{A2-proof}
For any $\delta>0$ and $\bar{z}=(\bar{x},u_-(\bar{x}),d_x u_-(\bar{x}))\in\Phi_H^{-\delta}(\mathrm{J}^1_{u_-})$, there is a sequence
\begin{equation}\label{p-convergence}
\{Z_n\}_{n\geq1}\subset\mathrm{J}^1_{\varphi}\quad \text{such that}\quad \lim_{n\rightarrow+\infty}\Phi^n_H(Z_n)=\bar{z}.
\end{equation}
\end{proposition}

\begin{proof}
We apply Theorem \ref{global-chm1} to find, for $n\geqslant 1, Z_n=(x_n, u_n, p_n)\in\mathrm{J}^1_{\varphi}$ such that the corresponding characteristic segments $\Phi^{t}_H(Z_n)=(x_n(t), u_n(t), p_n(t)), t\in[0,n]$ satisfies
\begin{equation}\label{eq:4}
u_n(t)=T^-_t \varphi(x_n(t)),\quad x_n(n)=\bar{x}.
\end{equation}
Instead of $\Phi^{t}_H(Z_n)$, we consider their time translations
\[
\hat{z}_n(t)=(\hat{x}_n(t),\hat{u}_n(t),\hat{p}_n(t)):=(x_n(t+n), u_n(t+n), p_n(t+n)),\quad t\in[-n,0].
\]
Again by uniform Lipschitz property of $\{T^-_t \varphi\}_{t\geq0}$, the characteristic segments $\{\hat{z}_n(t):t\in[0,n]\}_{n\geqslant1}$ are uniformly bounded and, since they are solutions of \eqref{ch}, thus pre-compact. For an arbitrary convergent subsequence $\hat{z}_{n_j}(t)$, we define
\[
\bar{z}(t)=(\bar{x}(t),\bar{u}(t),\bar{p}(t)):=\lim_{j\rightarrow\infty}(\hat{x}_{n_j}(t),\hat{u}_{n_j}(t),\hat{p}_{n_j}(t)),\quad t\leqslant0.
\]
In particular, $\bar{x}_{n}(0)=\bar{x}$. It follows that $\bar{z}(t)$ is an orbit of \eqref{ch}. On the other hand, \eqref{eq:4} as well as the convergence of solution semigroup imply that for  and $t\leqslant0$,
\begin{equation}\label{u-convegence}
\bar{u}(t)=\lim_{j\rightarrow\infty}\hat{u}_{n_j}(t)=\lim_{j\rightarrow\infty}T^-_{n_j}\varphi(\hat{x}_{n_j}(t))=u_-(\bar{x}(t)).
\end{equation}
Thus for any limit $\bar{z}(t),t\leqslant0$, we have
\[
\pi\circ\bar{z}(t)\subset J^0_{u_-}.
\]
Now we apply Lemma \ref{proof-A2-p} to conclude that $\bar{p}(0)\in D^\ast u_-(\bar{x}(0))=D^\ast u_-(\bar{x})$. Since $\bar{z}\in\Phi_H^{-\delta}(\mathrm{J}^1_{u_-})$, one deduces from Proposition \ref{cali-pro} (1) that $u_-$ is differentiable at $\bar{x}$. Thus $\bar{p}(0)=d_x u_-(\bar{x})$ and, with the aid of continuous dependence theorem of ODE, this proves the convergence of the whole sequence $\hat{z}_n(t)$ (It is clear that $\hat{z}_n(t)$ is a $(u_-,L)$-calibrated curve). Now we finish by noticing that
\[
\lim_{n\rightarrow\infty}\Phi^{n}_H(Z_n)=\lim_{n\rightarrow\infty}\hat{z}_n(0)=\bar{z}(0)=(\bar{x},u_-(\bar{x}),d_x u_-(\bar{x})).
\]
\end{proof}


\section{Proof of Theorem \ref{thm2}}
Our second main theorem concerns the \textbf{oriented connectibility} of $\widetilde{\mathcal{N}}_{u_-}, \widetilde{\mathcal{N}}_{v_+}$ associated to different weak KAM solutions $u_-\in\mathcal{S}_-$ and $v_+\in\mathcal{S}_+$. By Definition \ref{Mane-slice}, $\widetilde{\mathcal{N}}_{u_-}, \widetilde{\mathcal{N}}_{v_+}$ are compact $\Phi^t_H$-invariant sets. So $\Phi^t_H$-orbits connecting them, if exist, must be complete and asymptotic to each end in infinite time. Here, the notion ``oriented'' means the choice of $\alpha$-(thus also $\omega$-)limit set of the connecting orbit is fixed. We recall the assumptions, which hold throughout this section: there is $\varphi\in C(M,\R)$ such that
\begin{equation}\label{eq:thm2-convergence}
\lim_{t\to +\infty}T_t^- \varphi=u_-, \quad  \lim_{t\to +\infty} T_t^+ \varphi= v_+,
\end{equation}
and the limits are ordered by
\begin{equation}\label{eq:thm2-order}
v_+<u_-.
\end{equation}
For convenience, we set
\begin{equation}\label{u-v}
\eps_0:=\frac{1}{4}\min_{x\in M}\{u_-(x)-v_+(x)\}>0.
\end{equation}

\subsection{Proof of \textbf{(B1)}}
This part contains a proof of non-existence of the $\Phi^t_H$-orbits approaching $\widetilde{\mathcal{N}}_{u_-}$ in the infinite future and to $\widetilde{\mathcal{N}}_{v_+}$ in the infinite past. As a first step, we notice that the assumptions \eqref{eq:thm2-convergence} and \eqref{eq:thm2-order} implies

\begin{lemma}\label{lem:heteroclinic orbit-0}
For any $\eps\in(0,\eps_0)$, there exists $\varphi_\eps \in C^\infty(M,\mathbb{R})$ satisfying
\begin{equation}\label{eq:varphi-epsilon1}
\text{either}\quad v_+<\varphi_\eps<v_+ +\eps,\quad\quad\text{or}\quad u_--\eps<\varphi_\eps<u_-
\end{equation}
and the convergence condition
\begin{equation}\label{eq:varphi-epsilon2}
\lim_{t\to +\infty}T_t^- \varphi_\eps=u_-, \quad  \lim_{t\to +\infty} T_t^+ \varphi_\eps= v_+.
\end{equation}
\end{lemma}

\begin{proof}
We shall focus on the first case in \eqref{eq:varphi-epsilon1}, the second is obtained by a similar way. Using the convergence assumption \eqref{eq:thm2-convergence} and \eqref{u-v}, for any $\eps\in(0,\eps_0)$, there exist $\tau_0,\tau_\eps$ such that
\begin{equation}\label{ieq:lemmaB-1}
T_{\tau_0}^- \varphi \geqslant u_--\eps_0>v_++\eps_0\quad\text{and}\quad  T_{\tau_\eps}^+ \varphi \leqslant v_++\frac{\eps}{2} <u_-.
\end{equation}
Combining Proposition \ref{prop-sg} (1) and Proposition \ref{A6}, we have
\begin{equation}\label{ieq:lemmaB-2}
v_+=T_{\tau_\eps+\tau_0}^+v_+<T_{\tau_\eps+\tau_0}^+\big(v_+ +\eps_0\big)
<T_{\tau_\eps+\tau_0}^+\circ T_{\tau_0}^-\varphi=T_{\tau_\eps}^+\circ\big(T_{\tau_0}^+\circ T_{\tau_0}^-\varphi\big)\leqslant T_{\tau_\eps}^+\varphi<u_-.
\end{equation}
and the sequence of inequalities
\begin{align*}
u_-&\,=\lim_{s\to +\infty}T_s^- \circ T_{\tau_0}^-\varphi\\
&\,\leqslant\lim_{s\to +\infty} T_s^- \circ T_{\tau_0+\tau_\eps }^- \circ T_{\tau_0+\tau_\eps}^+ \circ T_{\tau_0}^- \varphi\\
&\,=\lim_{t\to +\infty} T_t^-\circ T_{\tau_0+\tau_\eps}^+ \circ T_{\tau_0}^- \varphi\leqslant u_-.
\end{align*}
holds uniformly on $M$. Similarly,
\[
v_+\leqslant \lim_{t\to +\infty} T_t^+ \circ T_{\tau_0+\tau_\eps}^+ \circ T_{\tau_0}^- \varphi\leqslant\lim_{t\to +\infty} T_t^+ \circ T_{\tau_\eps}^+ \varphi= v_+.
\]
Summarizing the above calculation, we have
\begin{equation}\label{eq:pf-heteroclinic orbit-0}
\lim_{t\to +\infty} T_t^-\big(T_{\tau_0+\tau_\eps}^+\circ T^-_{\tau_0} \varphi\big)=u_-, \quad \lim_{t\to +\infty} T_t^+ \big( T_{\tau_0+\tau_\eps}^+ \circ T_{\tau_0}^- \varphi\big)=v_+.
\end{equation}
We choose \textbf{any} $\varphi_\eps\in C^\infty(M,\R)$ satisfying
$$
T_{\tau_0+\tau_\eps}^+\circ T_{\tau_0}^- \varphi\leqslant\varphi_\eps\leqslant T_{\tau_0+\tau_\eps}^+ \circ T_{\tau_0}^- \varphi +\frac{\eps}{2}.
$$
From \eqref{ieq:lemmaB-1} and \eqref{ieq:lemmaB-2}, it is easy to check that $\varphi_\eps$ satisfies \eqref{eq:varphi-epsilon1}. To verify \eqref{eq:varphi-epsilon2}, it is necessary to employ \eqref{u-v} and \eqref{ieq:lemmaB-2} to see that
\begin{align*}
T_t^-\big(T_{\tau_0+\tau_\eps}^+\circ T^-_{\tau_0} \varphi\big)&\leqslant  T_t^- \varphi_\eps\leqslant T_t^-u_-,\\
T_t^+\big(T_{\tau_0+\tau_\eps}^+\circ T^-_{\tau_0} \varphi\big)&\leqslant  T_t^+ \varphi_\eps<T_t^+(u_--\eps_0)\leqslant T_t^+\circ T^-_{\tau_0}\varphi\leqslant T_{t-\tau_0}^+ \varphi,
\end{align*}
Sending $t$ to $+\infty$ and employing \eqref{eq:pf-heteroclinic orbit-0}, we complete the proof.
\end{proof}

\vspace{1em}
The proof of (B1) relies on a simple but useful observation motivated by Lemma \ref{lem:heteroclinic orbit-0}.
\begin{proposition}\label{prop-B2}
Assume $z_0=(x_0,u_0,p_0)$ satisfies
\begin{equation}\label{ieq:prop-B2}
v_+(x_0)<u_0,
\end{equation}
then either $\omega(z_0)=\emptyset$ or for any $\bar{z}=(\bar{x},\bar{u},\bar{p})\in\omega(z_0), u_-(\bar{x})\leqslant\bar{u}$. Similarly, assume
\[
u_0<u_-(x_0),
\]
then either $\alpha(z_0)=\emptyset$ or for any $\bar{z}=(\bar{x},\bar{u},\bar{p})\in\alpha(z_0), \bar{u}\leqslant v_+(\bar{x})$.
\end{proposition}

\begin{proof}
We shall focus on the first case and assume $\omega(z_0)\neq\emptyset$, otherwise there is nothing to prove. Fixing the notation $\Phi^t_H(z_0)=z(t)=(x(t),u(t),p(t))$, it follows from \eqref{ieq:prop-B2} that
\begin{equation}\label{eq:prop-4.3-pf-0}
v_+(x(0))< u(0).
\end{equation}
By the definition of $\omega(z_0)$, there is a positive sequence $\{t_n\}_{n\geqslant1}$ such that
\[
\lim_{n\to +\infty}t_n=+\infty,\quad\text{and}\quad\lim_{n\to +\infty}(x(t_n),u(t_n))=(\bar{x},\bar{u}).
\]
For $0<\eps<u(0)-v_+(x(0))$, we use Lemma \ref{lem:heteroclinic orbit-0} to find $\varphi_{\eps} \in C^\infty(M,\mathbb{R})$ satisfying
\begin{equation}\label{eq:proof-B2}
\varphi_{\eps}\leqslant v_++\eps,\quad\text{and}\quad\lim_{t\to +\infty} T_t^-\varphi_{\eps}=u_-.
\end{equation}
As a consequence, it is obvious that $\varphi_{\eps}(x(0))<u(0)$. By the definition of solution semigroup, we have
\begin{align*}
T_{t_n}^-\varphi_{\eps}(x(t_n))&\,=\inf_{x'\in M}h_{x',\varphi_{\eps}(x')}(x(t_n), t_n)\leqslant h_{x(0),\varphi_{\eps}(x(0))}(x(t_n),t_n)\\
&\,< h_{x(0),u(0)}(x(t_n),t_n)\leqslant u(t_n),
\end{align*}
where the last inequality follows from Corollary \ref{Minimality}. Using \eqref{eq:proof-B2}, we find
\[
u_-(\bar{x})=\lim_{n\rightarrow\infty}T_{t_n}^-\varphi_{\eps}(x(t_n))\leqslant\lim_{n\rightarrow\infty}u(t_n)=\bar{u}.
\]
\end{proof}

We proceed by supposing that there exists $z_0=(x_0,u_0,p_0)\in J^1(M,\R)$ such that $\alpha(z_0)\cap\widetilde{\mathcal{N}}_{u_-}\neq\emptyset$. To prove the conclusion (B1), it suffice to verify the assumption of Proposition \ref{prop-B2}. However, using the same notations as before, the definition of $\alpha(z_0)$ ensures some $\mathcal{T}<0$ such that
\[
v_+(x(\mathcal{T}))<u_-(x(\mathcal{T}))-\eps_0\leqslant u(\mathcal{T}).
\]
Then we may replace $z_0$ by $\Phi^{\mathcal{T}}_H(z_0)$ and apply \ref{prop-B2} since $\omega(z_0)=\omega(\Phi^t_Hz_0)$.

\begin{remark}
By Proposition \ref{prop-B2}, we can conclude that there does not exist any finite chain of points $\{z_i\}_{0\leqslant i\leqslant N}\subset J^1(M,\R)$ with $z_0\in\widetilde{\mathcal{N}}_{u_-}$ and $z_N\in\widetilde{\mathcal{N}}_{v_+}$ such that for each $0\leqslant i\leqslant N-1, \Phi^t_H(z_i)$ is defined by all $t\in\R, \omega(z_{i})\cap\alpha(z_{i+1})\neq\emptyset$.
\end{remark}

\subsection{Proof of \textbf{(B2)}}
In this part, we shall show how to employ Theorem A and tools developed in its proof to construct heteroclinic orbits $\Phi^t_H(Z_0),t\in\R$ such that $\alpha(Z_0)\subset\widetilde{\mathcal{N}}_{v_+}, \omega(Z_0)\subset\widetilde{\mathcal{N}}_{u_-}$. We begin by introducing the following construction: for any $\eps\in(0,\eps_0)$ and $\varphi_\eps$ given by Lemma \ref{lem:heteroclinic orbit-0}, we apply (A1) to produce a semi-infinite orbit
\[
z_\eps(t)=\Phi^t_H(z_\eps(0)):=(x_\eps(t),u_\eps(t),p_\eps(t)),\quad t\in [0, +\infty)
\]
with the conditions $z_\eps(0)\in \mathrm{J}^1_{\varphi_\eps}, \omega(z_\eps(0))\subset\widetilde{ \mathcal{N}}_{u_-}$ and for any $t\geqslant0$,
\begin{equation}\label{eq:heteroclinic-pf-1}
u_\eps(t)=T_t^-\varphi_\eps(x_\eps(t)),\quad p_\eps(t)\in D^\ast_xT_t^-\varphi_\eps(x_\eps(t)).
\end{equation}
Here we use \eqref{semi-infi-ch} in the proof of (A1). Let us define
$$
w=\frac{1}{2}(u_-+v_+).
$$
We combine \eqref{u-v} and \eqref{eq:varphi-epsilon1} to deduce
\begin{equation}\label{eq:proof-order}
\varphi_\eps<v_+ +\eps<v_+ +\eps_0\leqslant w<u_-.
\end{equation}
Due to the fact $\lim_{t\to +\infty} T_t^- \varphi_\eps=u_-$, there exists $t_\eps>0$ such that
\begin{equation}\label{eq:heteroclinic-pf-1}
w(x_\eps(t_\eps))=T_{t_\eps}^-\varphi_\eps(x_\eps(t_\eps))=u_\eps(t_\eps).
\end{equation}
It is readily seen that
\begin{lemma}\label{t-epsilon-infi}
$\lim_{\eps \to 0} t_\eps =+\infty$.
\end{lemma}

\begin{proof}
By the expansiveness estimate of solution semigroup,
\begin{equation}\label{eq:heteroclinic orbit-1}
\|T^-_{t_\eps} \varphi_\eps -v_+\|_\infty \leqslant e^{\lambda t_\eps }\| \varphi_\eps -v_+  \|_\infty \leqslant e^{\lambda t_\eps } \eps,
\end{equation}
and from \eqref{eq:proof-order} and \eqref{eq:heteroclinic-pf-1} we deduce that
\begin{equation}\label{eq:heteroclinic orbit-2}
\begin{split}
\| T^-_{t_\eps} \varphi_\eps -v_+\|_\infty  \geqslant &\, T^-_{t_\eps} \varphi_\eps(x_\eps(t_\eps) ) -v_+ (x_\eps(t_\eps) ) \\
 =&\,  u_\eps(t_\eps)-v_+(x_\eps(t_\eps) )\\
 =&\,w(x_\eps(t_\eps) )-v_+(x_\eps(t_\eps) )\geqslant\eps_0>0.
\end{split}
\end{equation}
Combining \eqref{eq:heteroclinic orbit-1} and \eqref{eq:heteroclinic orbit-2},
\begin{equation}\label{eq:heteroclinic orbit-3}
t_\eps \geqslant \frac{1}{\lambda} \ln \frac{\eps_0}{\eps}, \quad \forall \eps>0.
\end{equation}
which implies that $\displaystyle \lim_{\eps \to 0} t_\eps =+\infty$.
\end{proof}

\vspace{1em}
Now we define $(X_\eps(t),U_\eps(t), P_\eps(t)) $ as
$$
Z_\eps(t)=(X_\eps(t),  U_\eps(t), P_\eps(t)):=\Big( x_\eps(t+t_\eps),u_\eps(t+ t_\eps),p_\eps(t+t_\eps) \Big), \quad t\in [-t_\eps,+\infty).
$$
It is readily seen that $ \{Z_\eps(0)=(X_\eps(0), U_\eps(0),P_\eps(0)) \}_{\eps \in (0,\eps_0) }$ is bounded. To show this, it is necessary to focus on the boundedness of $U_\eps(0)$ and $P_\eps(0)$. By the equation \eqref{eq:heteroclinic-pf-1},
\begin{equation}\label{u-epsilon}
U_\eps(0)=u(t_\eps) = w (x_\eps(t_\eps))
\end{equation}
This implies that for any $\eps\in(0,\eps_0)$,
$$
|U_\eps(0) | \leqslant \|w \|_\infty\leqslant \frac{\|u_- \|_\infty+ \|v_+ \|_\infty}{2}.
$$
By Theorem \ref{global-chm}, $P_\eps(0)\in D^*_x T_{t_\eps}^-\varphi_\eps (x(t_\eps)) $. Then $P_\eps(0)$ is bounded for any $\eps\in(0,\eps_0)$ is a consequence of

\begin{lemma}\label{B1-equi-Lip}
$\{T_t^- \varphi_\eps (x)\}_{t>1,\eps\in(0,\eps_0)}$ is equi-Lipschitz on $M$.
\end{lemma}

\begin{proof}
Notice that by \eqref{eq:varphi-epsilon2}, then there exists $K_1>0$ such that
\begin{equation}\label{eq:pf-lem-heteroclinic orbit-2}
|T_t^- \varphi_\eps (x) | \leqslant K_1 , \quad \forall (x,t, \eps ) \in M\times [0, +\infty) \times (0,\eps_0).
\end{equation}
Then for any $x,y\in M$ and $t>1$,
\begin{align*}
|T_t^-\varphi_\eps(x)-T_t^-\varphi_\eps(y)|&\,=\left| \inf _{z \in M} h_{z,T_{t-1}^- \varphi_\eps(z)} (x,1) -  \inf_{z \in M} h_{z,T_{t-1}^- \varphi_\eps(z)} (y,1) \right| \\
&\,\leqslant  \sup_{z \in M} \left| h_{z,T_{t-1}^- \varphi_\eps(z)} (x,1) - h_{z,T_{t-1}^- \varphi_\eps(z)} (y,1)\right|\leqslant c_0 \cdot d(x,y),
\end{align*}
where $c_0>0$ denotes the Lipschitz constant of the function
\[
(x_0,u_0,x)\mapsto h_{x_0,u_0}(x,1)
\]
on $M \times [-K_1,K_1] \times M$. Using \eqref{eq:pf-lem-heteroclinic orbit-2}, the above inequality holds for any $\eps\in(0,\eps_0)$.
\end{proof}
Therefore, $P_\eps(0)$ is bounded by $c_0$ for any $\eps\in(0,\eps_0)$. Thus by the pre-compactness of $ \{Z_\eps(0) \}_{\eps \in (0,\eps_0)}$, there exists a sequence $\{\eps_n\}$ converging to $0$ such that
$$
\lim_{\eps_n \to 0 } (X_{\eps_n}(0), U_{\eps_n}(0),P_{\eps_n}(0))=(X_0,U_0,P_0):=Z_0
$$
and the orbit we are looking for is defined as
\begin{equation}\label{eq:}
Z_0(t):=\Phi_H^t (Z_0)=(X_0(t),U_0(t),P_0(t)).
\end{equation}
Notice that $\lim_{n\rightarrow\infty}t_{\eps_{n}}=+\infty$ as $\lim_{n\rightarrow\infty}\eps_{n}=0$, and the sequence of orbits $Z_{\eps_{n}}:[-t_{\eps_{n}},+\infty)\rightarrow J^1(M,\R)$ converges uniformly to $Z_0$ on compact intervals, so $Z_0(t)$ is complete, i.e., is defined for all $t\in\R$. Based on our construction, it is not hard to show

\begin{lemma}
Both $\omega(Z_0)$ and $\alpha(Z_0)$ are non-empty.
\end{lemma}
	
\begin{proof}
For any $\eps\in(0,e^{-\lambda} \eps_0)$, \eqref{eq:heteroclinic orbit-3} implies $t_\eps>1$. Applying Lemma \ref{B1-equi-Lip}, we get    $U_\eps(t), t\in [-t_\eps,+\infty) $ is bounded by $K_1$   and $P_\eps(t), t\in [-t_\eps,+\infty) $ is bounded by $c_0$. Thus $\{Z_\eps(t):t\geqslant t_\eps\}_{\eps\in(0,e^{-\lambda} \eps_0)}$ is uniformly bounded. As the $C^0$ limit of $Z_{\eps_n}$, $Z_0:\R\rightarrow J^1(M,\R)$ is also bounded and the conclusion follows.
\end{proof}

Now we are ready to show that	
\begin{proposition}\label{prop:B1-omega}
$\omega(Z_0)\subset\{(x,u,p):u=u_-(x)\}$ and $\alpha(Z_0)\subset\{(x,u,p):u=v_+(x)\}$.
\end{proposition}

\begin{proof}
Notice that by \eqref{u-epsilon},
\[
U_{\eps}(0)=w(x_\eps(t_\eps))>v_+(x_\eps(t_\eps))+\eps_0=v_+(X_\eps(0))+\eps_0.
\]
By taking $\eps=\eps_n$ and sending $n$ to infinity, we obtain that
\[
U_0(0)\geqslant v_+(X_0(0))+\eps_0.
\]
Now the first part of Proposition \ref{prop-B2} allow us to conclude that $\omega(Z_0)\subset\{(x,u,p):u\geqslant u_-(x)\}$. On the other hand, from Lemma \ref{lem:heteroclinic orbit-0}, it is easy to check that $\varphi_\eps<u_-$. Thus by Proposition \ref{prop-sg} (1), we have for any $t\in[0,+\infty)$,
\begin{align*}
&\,U_\eps(t)=u_\eps(t+t_\eps)=T_{t+t_\eps}^-\varphi_\eps(x_\eps(t+t_\eps))\\
<&\,T_{t+t_\eps}^-u_-(x_\eps(t+t_\eps))=u_-(x_\eps(t+t_\eps))=u_-(X_\eps(t)).
\end{align*}
Fixing $t\geqslant0$, taking $\eps=\eps_n$ and sending $n$ to infinity as before, we have $U_0(t)\leqslant u_-(X_0(t))$ for each $t\geqslant0$. Thus $\omega(Z_0)\subset\{(x,u,p):u\leqslant u_-(x)\}$ and the first conclusion holds.

\vspace{1em}
For the second conclusion, one checks the fact that
\[
U_0(0)\leqslant u_-(X_0(0))-\eps_0
\]
in the same way as before and apply the other part of Proposition \ref{prop-B2} to obtain $\alpha(Z_0)\subset\{(x,u,p):u\leqslant v_+(x)\}$. For a fixed $t\leqslant0$ and $t_\eps$ large enough,
\begin{align*}
&\,U_\eps(t)=u_\eps(t+t_\eps)=T_{t+t_\eps}^-\varphi_\eps(x_\eps(t+t_\eps))\\
>&\,T_{t+t_\eps}^-v_+(x_\eps(t+t_\eps))\geqslant v_+(x_\eps(t+t_\eps))=v_+(X_\eps(t)),
\end{align*}
where the first inequality use that fact $\varphi_\eps>v_+$ and Proposition \ref{prop-sg} (1), the second inequality follows from Proposition \ref{prop-sg} (2) since forward weak KAM solutions are subsolutions to \eqref{hj}. The proof is completed by repeating the limiting argument as before.
\end{proof}

To finish the proof of (B2), we invoke Lemma \ref{proof-A2-p} and the subsequent Remark \ref{proof-A2-p'} to see that $\alpha(Z_0)\subset\mathrm{J}^1_{v_+}, \omega(Z_0)\subset\mathrm{J}^1_{u_-}$. The $\Phi^t_H$-invariance of $\alpha(Z_0)$ and $\omega(Z_0)$ implies that $\alpha(Z_0)\subset\widetilde{\mathcal{N}}_{v_+}$ and $\omega(Z_0)\subset\widetilde{\mathcal{N}}_{u_-}$.

\section{Variational aspects of the constructed orbits}
In this section, we explore action minimizing properties of the asymptotic orbits constructed in the main theorems and establish a variational characterization on the $\alpha/\omega-$limit of such orbits. We also present a proposition relating such characterization to the sets that appear in Definition \ref{Mane-slice}.

\vspace{0.5em}
With the variational principle introduced in the Section 2.1, we start by formulating the notion of action minimizers in contact Hamiltonian systems.
\begin{definition}\label{global-mini}
Let $I\subset\R$ be an interval. A locally Lipschitz curve $(x(\cdot),u(\cdot)):I\rightarrow J^0(M,\R)$ is an \textbf{action minimizer over $I$} if for any $a,b\in I, a\leqslant b$,
\begin{equation}\label{minimizer}
u(b)=h_{x(a),u(a)}(x(b), b-a).
\end{equation}
If $I=[c,+\infty)($resp. $(-\infty,c])$, we call $(x(\cdot),u(\cdot))$ a \textbf{forward (resp. backward) action minimizer}; if $I=\R$, such a minimizer is called a \textbf{global action minimizer}.
\end{definition}

Combining Definition \ref{Implicit variational} and Proposition \ref{fundamental-prop}, the above definition implies
\begin{lemma}\label{prop-act-minimizer}
With the same notations as Definition \ref{global-mini}.
\begin{enumerate}[(1)]
  \item Assume $(x(\cdot),u(\cdot)):I\rightarrow M$ is an action minimizer over $I$, $J\subset I$ is an interval, then $(x(\cdot),u(\cdot))$ is also an action minimizer over $J$;

  \item Assume $x:[0,t]\rightarrow M$ is a Lipschitz minimizer of $h_{x_0,u_0}(x,t)$ and set $u(s)=h_{x_0,u_0}(x(s),s)$, then $(x(\cdot),u(\cdot)):[0,t]\rightarrow M$ is an action minimizer over $[0,t]$;

  \item Let $(x(\cdot),u(\cdot)):I\rightarrow J^0(M,\R)$ be an action minimizer over $I$, then for any $a<b\in I$,
      \begin{equation}\label{action-minimizer1}
      \hat{x}(t):=x(t+a),\quad t\in[0, b-a]
      \end{equation}
      is a minimizer of $h_{x(a),u(a)}(x(b), b-a)$. Moreover, if $a,b$ are not endpoints of $I$, then the minimizer of $h_{x(a),u(a)}(x(b), b-a)$ is \textbf{unique}.

  \item Assume for each $n\in\mathbb{N}, (x_n(\cdot),u_n(\cdot)):I_n\rightarrow J^0(M,\R)$ is an action minimizer over $I_n$ and $I_n\subset I_{n+1}$. Let $I=\cup_n I_n$ and $(x_n(\cdot),u_n(\cdot))$ converges to a locally Lipschitz curve $(x(\cdot),u(\cdot)):I\rightarrow J^0(M,\R)$ uniformly on compact intervals of $I$, then $(x(\cdot),u(\cdot))$ is also an action minimizer over $I$.
\end{enumerate}
\end{lemma}

It is direct to verify the asymptotic orbit found in (A2) is forward action minimizing by
\begin{proposition}\label{A2-act-mini}
Assume there is $\varphi\in C(M,\R)$ and $z:[0,+\infty)\rightarrow J^1(M,\R)$ is a $\Phi^t_H$-orbit such that
\[
z(t)=(x(t),u(t),p(t)),\quad u(t)=T^-_t\varphi(x(t)),\quad t\in[0,+\infty),
\]
then $(x(\cdot),u(\cdot))$ is a forward action minimizer.
\end{proposition}

\begin{proof}
By semi-group property of $\{T^-_t\}_{t\geqslant0}$, for any $0\leqslant a\leqslant b$,
\begin{align*}
u(b)=\,&T^-_{b}\varphi(x(b))=T^-_{b-a}(T^-_a\varphi)(x(b))=\inf_{x\in M}h_{x,T^-_a\varphi(x)}(x(b),b-a)\\
\leqslant\,&h_{x(a),T^-_a\varphi(x(a))}(x(b),b-a)=h_{x(a),u(a)}(x(b), b-a)\leqslant u(b),
\end{align*}
where last inequality is deduced from Corollary \ref{Minimality}.
\end{proof}

Notice that Lemma \ref{prop-act-minimizer} (3) helps us to lift action minimizers to $\Phi^t_H$-orbits in a unique way. From Definition \ref{global-mini} and \eqref{action-minimizer1}, it follows that for $t\in[0,b-a]$,
\[
\hat{u}(t):=h_{x(a),u(a)}(\hat{x}(t), t)=h_{x(a),u(a)}(x(a+t),t)=u(a+t).
\]
By defining the $p$-component as
\[
\hat{p}(t):=\frac{\partial L}{\partial\dot{x}}(\hat{x}(t),\hat{u}(t),\dot{\hat{x}}(t))=\frac{\partial L}{\partial\dot{x}}(x(t+a),u(t+a),\dot{x}(t+a)),
\]
we invoke Remark \ref{minimizer-orbit} to see that
\[
\hat{z}(t):=(\hat{x}(t),\hat{u}(t),\hat{p}(t)),\quad t\in[0,b-a]
\]
is $C^1$ and satisfies \eqref{ch}. Exhausting $I$ by increasing compact intervals and apply (1), one can show that

\begin{lemma}\label{act-mini-orbit}
For any action minimizer $(x(\cdot),u(\cdot)):I\rightarrow J^0(M,\R)$,
\[
z(t)=(x(t),u(t),p(t)),\quad p(t)=\frac{\partial L}{\partial\dot{x}}(x(t),u(t),\dot{x}(t)),\quad t\in I
\]
is the \textbf{unique} $C^1$ orbit of $\Phi^t_H$ such that $\pi\circ z(t)=(x(t),u(t))$. We shall call such $z:I\rightarrow J^1(M,\R)$ an \textbf{action minimizing orbit over $I$}.

\vspace{1em}
By adopting the convention in Definition \ref{global-mini}, we say $z$ is a \textbf{global (resp. forward, backward) action minimizing orbit} if $I=\R\,($resp. $[c,+\infty), (-\infty,c])$. We introduce the notations
\begin{equation*}
\widetilde{\mathcal{G}}=\{z_0\in J^1(M,\R):z(t)=\Phi^t_H(z_0)\text{ is a global action minimizing orbit}\},\quad \mathcal{G}=\pi\widetilde{\mathcal{G}}.
\end{equation*}
\end{lemma}

Now we are ready to show
\begin{proposition}\label{A2-B2-act-minimizer}
Any orbit constructed in the proof of (B2) is a global action minimizer.
\end{proposition}

\begin{proof}
We shall employ the notations in Section 4.2. Assume $\Phi^t_H(Z_0)=(X_0(t),U_0(t), P_0(t)), t\in\R$ is an orbit constructed in the proof of (B2). It is a $C^0$-limit of some subsequence of $\Phi^t_H$-orbits
\[
\{(X_\eps(t),U_\eps(t), P_\eps(t)), t\in[-t_\eps,+\infty)\}_{\eps>0}
\]
satisfying $U_\eps(t)=T^-_{t+t_\eps}\varphi_\eps(X_\eps(t))$ for some $\varphi_\eps\in C^\infty(M,\R)$. By Proposition \ref{A2-act-mini}, $(X_\eps(t),U_\eps(t))$ is an action minimizer over $[-t_\eps,+\infty)$ with a common Lipschitz constant with respect to $\eps$. Thus due to Lemma \ref{act-mini-orbit}, for any $a,b,t\in [-t_\eps,+\infty ), a\leqslant b$,
\begin{equation}\label{minimizer-1}
\begin{split}
U_\eps(b)=&\, h_{X_\eps (a),U_\eps (a)}(X_\eps (a), b-a).\\
P_\eps(t)=&\,   \frac{\partial L}{\partial\dot{x}}( X_\eps(t), U_\eps(t), \dot X_\eps(t)).
\end{split}
\end{equation}
For any given $a<b\in\R$, by applying Lemma \ref{t-epsilon-infi}, we can always assume $-t_\eps<a$ for sufficiently small $\eps$. We send $\eps$ to $0$, and use \eqref{minimizer-1} as well as the continuity of $(x_0,u_0,x)\mapsto h_{x_0,u_0}(x,b-a)$ for any $a<b$ to conclude that
\begin{equation}\label{minimizer-2}
U_0(b)=h_{X_0 (a),U_0 (a)}(X_0 (a), b-a).
\end{equation}
Thus $(X_0(t),U_0(t))$ is a global action minimizer. Since $(X_0(t),U_0(t), P_0(t))$ is a $\Phi^t_H$-orbit, the convergence of $(X_\eps(\cdot), U_\eps(\cdot), P_\eps(\cdot))$ and Lemma \ref{act-mini-orbit} imply that
\[
P_0(t)=\frac{\partial L}{\partial\dot{x}}( X_0(t), U_0(t), \dot X_0(t)).
\]
Now the arbitrariness of $a,b$ shows that $(X_0(t),U_0(t),P_0(t))$ is global minimizing.
\end{proof}

Now we include a general statement, originally formulated in \cite{XYZ}, giving a characterization of limit sets of forward/backward action minimizing orbits via action functions. Our version is slightly adapted from the original form and the proof is only sketched, the readers can also consult the \cite{XYZ} for more details.

\begin{proposition}\cite[Proposition 1.1]{XYZ}
Assume $z(t)=(x(t),u(t),p(t)), t\geqslant 0$ is a forward action minimizing orbit and there exists $u_-\in\mathcal{S}^-$ such that $\lim_{t\rightarrow+\infty}|u(t)-u_-(x(t))|=0$, then for any $\Phi^t_H$-orbit $\bar{z}(t)=(\bar{x}(t),\bar{u}(t),\bar{p}(t))\in\omega(z(0))$ and real numbers $a\leqslant b$,
\begin{equation}\label{semi-static-eq-1}
\bar{u}(b)= \inf_{s>0}h_{\bar{x}(a),\bar{u}(a)}(\bar{x}(b),s).
\end{equation}
\end{proposition}

\begin{proof}
Notice that the forward action minimizing $z(t)=(x(t),u(t),p(t)), t\geqslant0$ satisfying
\begin{equation}\label{u-convergence}
\lim_{t\rightarrow\infty}|u(t)-u_-(x(t))|=0
\end{equation}
is bounded, we conclude $\omega (z(0))$ is not empty. For $(\bar{x}(0),\bar{u}(0),\bar{p}(0))\in\omega(z(0))$, there is a sequence $t_n\in\R$ with $\lim_{n\rightarrow\infty}t_n\rightarrow\infty$ such that $(x(t_n),u(t_n),p(t_n))\to(\bar{x},\bar{u},\bar{p})$ as $n\rightarrow\infty$. Let $\bar{z}(t)=(\bar{x}(t),\bar{u}(t),\bar{p}(t)):=\Phi_H^t(\bar{x},\bar{u},\bar{p}),\ t\in\R$. By applying Proposition \ref{prop-act-minimizer} (4), one easily deduce that $(\bar{x}(\cdot),\bar{u}(\cdot))$ is a global action minimizer, thus
\begin{equation}\label{eq:pf-cor-2}
\bar{u}(b)=h_{\bar{x}(a),\bar{u}(a)}(\bar{x}(b),b-a).
\end{equation}
For any $t\in\R$, by the assumption \eqref{u-convergence},
\begin{equation}\label{eq:pf-cor-1}
\bar{u}(t)=\lim_{n\to+\infty}u(t_n+t)= \lim_{n\to+\infty}u_-(x(t_n+t))=u_-(\bar{x}(t)).
\end{equation}
Thus we know that for any $b\geqslant a$ and $s>0$,
\begin{equation*}
\begin{split}
\bar{u}(b)=&\,u_-(\bar{x}(b))=T^-_{s}u_-(\bar{x}(t))=\inf_{z\in M}h_{z,u_-(z)}(\bar{x}(b),s)\\
  \leqslant&\,h_{\bar{x}(a),u_-(\bar{x}(a))}(\bar{x}(b),s)=h_{\bar{x}(a),\bar{u}(a)}(\bar{x}(b),s).
\end{split}
\end{equation*}
As a consequence of the above inequality and \eqref{eq:pf-cor-2}, we obtain \eqref{semi-static-eq-1} and complete the proof.		
\end{proof}

\begin{remark}
One can characterize the $\alpha$-limit set of a backward action minimizing orbit by reversing the time direction and replacing $u_-\in\mathcal{S}_-$ with $u_+\in\cS_+$. Then for any $\Phi^t_H$-orbit $\bar{z}(t)=(\bar{x}(t),\bar{u}(t),\bar{p}(t))\in\alpha(z(0))$ and real numbers $a\leqslant b$,
\begin{equation}\label{semi-static-eq-2}
\bar{u}(a)= \inf_{s>0}h^{\bar{x}(b),\bar{u}(b)}(\bar{x}(a),s).
\end{equation}
However, one notice that by Proposition \ref{fundamental-prop} (3), \eqref{semi-static-eq-1} and \eqref{semi-static-eq-2} are equivalent.
\end{remark}

The above discussions lead us to the notion of time-free action minimizing orbits. As we shall see in the next part, this class of action minimizers is closely related to the weak KAM solutions of \eqref{HJs}.
\begin{definition}\label{semi-static}
A locally Lipschitz curve $(x(\cdot),u(\cdot)):\R\rightarrow J^0(M,\R)$ is called semi-static if it is a global action minimizer and for any $a<b$,
\begin{equation}\label{semi-static curve}
u(b)=\inf_{s>0}h_{x(a),u(a)}(x(b),s)\Leftrightarrow u(a)=\sup_{s>0}h^{x(b),u(b)}(x(a),s).
\end{equation}
We use $\mathcal{N}$ to denote the union of the image of semi-static curves. By Lemma \ref{act-mini-orbit}, semi-static curves can be lifted to $\Phi^t_H$-orbits. We call them \textbf{semi-static orbits} and introduce
\[
\widetilde{\mathcal{N}}=\{z_0\in J^1(M,\R):z(t)=\Phi^t_H(z_0)\text{ is a semi-static orbit}\}.
\]
The set $\widetilde{\mathcal{N}}$ is called the \textbf{Ma\~{n}\'{e} set} associated to \eqref{ch}. It is evident that $\pi\widetilde{\mathcal{N}}=\mathcal{N}$.
\end{definition}

To the end of this section, we relate the concept of semi-static orbits defined above to the existence of weak KAM solutions to (HJs), and give a decomposition of Ma\~{n}\'{e} set by the Ma\~{n}\'{e} slices introduced in Definition \ref{Mane-slice}. This is summarized into
\begin{theorem}\cite[Theorem 1.3]{WWY4}\label{mane-nonempty-decomposition}
$\tilde{\mathcal{N}}\neq\emptyset$ is equivalent to either $\mathcal{S}_-$ or $\mathcal{S}_+$ is not empty. Moreover, $\widetilde{\mathcal{N}}$ is closed with the decomposition
\begin{equation}\label{decom-mane}
\widetilde{\mathcal{N}}=\bigcup_{u_-\in\mathcal{S}_-}\widetilde{\mathcal{N}}_{u_-}=\bigcup_{u_+\in\mathcal{S}_+}\widetilde{\mathcal{N}}_{u_+}.
\end{equation}
\end{theorem}

We believe that such a decomposition will be useful in future development of action minimizing method for contact Hamiltonian systems. We shall give a  complete proof of this theorem in the Appendix, see Section 7.2.

\section{Dynamics of global minimizing orbits in the model system}
In this last section, we shall apply our main results on a model system treated in our previous work \cite{JYZ}. This leads to an almost complete description of the dynamics of global minimizing orbits. In a concrete example, we discuss whether the heteroclinic orbits found in (B2) belong to the class of semi-static orbits. As a result, we show that, in autonomous contact Hamiltonian systems, global action minimizers are not always semi-static.

\subsection{Previous knowledge about the model systems}
Recall that the model Hamiltonian $H:J^1(M,\R)\rightarrow\mathbb{R}$ is written in the form
\begin{equation}\label{model}
H(x,u,p)=F(x,p)+\lambda(x)u,
\end{equation}
where $F\in C^\infty(T^\ast M,\R)$ is a Tonelli Hamiltonian in the classical sense and $\lambda\in C^\infty(M,\R)$.

\vspace{0.5em}
The study of the model systems began with the special case when $\lambda(x)$ is a positive constant, i.e., the so-called discounted Hamiltonian systems. Such systems appeared in the study of, among others, mechanical systems with a dissipation proportional to the velocity, the spin-orbit model celestial mechanics \cite{Ce}, deflation phenomena in economics \cite{K}, the optimal control problem \cite{Ben}, the transport model problem \cite{WL} and already attracted much research interests. Like the case when $H$ is independent of $u$, the first $2n$ equations of \eqref{ch} are independent of the last one. Thus one can forget the action variable $u$ to get a flow $\phi^t_H$ on $T^{\ast}M$, it follows that
\begin{equation}\label{con-sym}
(\phi^t_H)^\ast\Omega=e^{-\lambda t}\Omega,
\end{equation}
where $\Omega=dp\wedge dx$ is the canonical symplectic form on the cotangent bundle. Due to \eqref{con-sym}, $\phi^t_H$ is also called \textbf{conformal symplectic flow}.

\vspace{0.5em}
Early investigations from the dynamical viewpoint include the existence of periodic orbits \cite{Cas-periodic} and the \textbf{Birkhoff attractor} \cite{Le-Calvez1}-\cite{Le-Calvez2} for the dissipative twist maps, which is the time-$1$ map of flow $\phi^t_H$, on the two dimensional annulus. For the past two decades, the celebrated KAM theorem on the existence of quasi-periodic invariant tori was extended to the conformal symplectic flow and applied to the related model of celestial mechanics \cite{CCGL}-\cite{CCL2}. Simultaneously, action minimizing method, especially Aubry-Mather and weak KAM theory, was developed in \cite{MS} to detect the existence of the compact global attractor. We also want to mention that, the successful application of weak KAM theory for conformal symplectic flow allow \cite{DFIZ} to prove the existence of vanishing discount limit of the solution to Hamilton-Jacobi equations arising in \cite{LPV}.

\vspace{0.5em}
Recently in \cite{AF}, the uniqueness and relation between inner dynamics and symplectic aspects of the invariant submanifold of $\phi^t_H$, for instance what kind of dynamics on the invariant submanifold ensures they are isotropic, was explored and later, the dynamical study of conformal symplectic flow was extended to Hamiltonian flow on a \textbf{conformal symplectic manifolds} \cite{AA} following the geometric construction in \cite{Va1}-\cite{Va2}.

\vspace{0.5em}
Due to the analysis in Example \ref{toy-model}, the conformal symplectic flows can be regarded as a special case of the contact Hamiltonian flows with monotone Hamiltonian. More extensively, assume $\lambda(x)$ is no-where vanishing on $M$, then
\begin{enumerate}
  \item[$(+)$] either $\lambda(x)>0$ and $\Phi^t_H$ admits a maximal global attractor $\mathcal{K}_H$,
  \item[$(-)$] or $\lambda(x)<0$ and $\Phi^{-t}_H$ admits a maximal global attractor $\mathcal{K}_H$,
\end{enumerate}
with topological properties listed in Theorem \ref{mono}.

\begin{remark}
In an upcoming work \cite{AHV}, the authors propose a higher dimensional generalization of Birkhoff attractors for the time-$1$ of $\phi^t_H$ from view of symplectic geometry and continue the study of its properties. In particular, the authors showed that the Birkhoff attractor shares the same cohomology as the base manifold $M$, which is quite close to (3) in Theorem \ref{mono}. It would be interesting to compare the notion of Birkhoff attractor in \cite{AHV} and that of global attractor discussed here.
\end{remark}

One can also ask what happens in the degenerate case when $\lambda$ is non-negative on $M$ but vanishes on part of $M$. The extreme case is that $\lambda$ is identically zero on $M$ and the system becomes a suspended Tonelli Hamiltonian system, then the conservation of the volume of phase space forbids the existence of any compact global attractor. In this case, the dynamics of $\Phi^t_H$ becomes a classical and extensive topic and certainly can not be included in this short section. The existence of compact global attractor as well as other descriptions of the dynamics for the model system with non-negative $\lambda$ will be the theme of another work.

\vspace{1em}
In the following, we shall focus on the model Hamiltonian satisfying the fluctuation condition:
\begin{enumerate}
  \item[($\clubsuit$)] there exist $x_1,x_2\in M$ such that $\lambda(x_1)\lambda(x_2)<0$.
\end{enumerate}

\begin{remark}
The fluctuation condition means that the contact Hamiltonian is not monotone with respect to $u$, i.e. $dH(\mathcal{R})$ changes sign on the phase space. We believe that \eqref{model} is one of the most simple and typical models for the non-monotone contact Hamiltonian.
\end{remark}

We knew that to explore the dynamics of $\Phi^t_H$, an important step is to study the solution or the set of solutions to \eqref{hj}. In this aspect, the model systems satisfying ($\clubsuit$) differ significantly from monotone systems in losing uniqueness of the solution to \eqref{hj} as we have
\begin{proposition}\cite[Section 2.1]{JYZ}\label{cv}
For the equation \eqref{HJs} with Hamiltonian \eqref{model} satisfying ($\clubsuit$), there is $c(H)\in\R$, uniquely determined by $H$, such that \eqref{HJs} admits a solution if and only if
\[
c(H)\leqslant0.
\]
Furthermore, if $c(H)<0$, then \eqref{HJs} admits a strict subsolutions $\varphi\in C^\infty(M,\R)$.
\end{proposition}
When $c(H)=0$, our information on the structure of weak KAM solutions is not enough for us to give a clear picture on the dynamics of global action minimizing orbits. Thus we restrict ourselves into the non-critical case, i.e.,
\begin{enumerate}
  \item[($\blacklozenge$)] $c(H)<0$.
\end{enumerate}

\vspace{0.3em}
Using the strict monotonicity of the actions of $T^{\pm}_t$ on subsolutions, i.e., Proposition \ref{A6}, we choose a strict subsolution $\varphi$ to \eqref{HJs} and define
\begin{equation}\label{principle-sol1}
\bar{u}_-:=\lim_{t\rightarrow\infty}T^-_t\varphi,\quad\underline{u}_+:=\lim_{t\rightarrow\infty}T^+_t\varphi.
\end{equation}
It follows that
\begin{equation}\label{order-toymodel}
\underline{u}_+(x)<\varphi(x)<\bar{u}_-(x)\quad\text{for any }\,\,\,x\in M.
\end{equation}
Then it is clear that $\widetilde{\mathcal{N}}_{ \bar{u}_- }\cap\widetilde{\mathcal{N}}_{\underline{u}_+}=\emptyset$.

\vspace{0.5em}
To proceed, we recall the main results on the structure of the set of weak KAM solutions and large-time behavior of \eqref{HJe} with Hamiltonian \eqref{model} obtained in \cite{NW}. It is noteworthy that for their validity, the assumptions \textbf{(H1)-(H3)} can be considerably weaken. To warm up, notice that $\mathcal{S}_\pm$ inherit a natural order structure $\preceq$ from the order on $\R$, that is, $u_\pm\preceq v_\pm$ if and only if $u_\pm(x)\leq v_\pm(x)$ for all $x\in M$.

\begin{proposition}\cite[Theorem 1-3]{NW}\label{prop:Large-time behavior1}
There is a maximal element (thus unique) $u^{\max}_-$ of $(\cS_-,\preceq)$ and a minimal element $u^{\min}_+$ of $(\cS_+,\preceq)$. For any initial data $\varphi\in C(M,\R)$,
\begin{itemize}
  \item[(1)] If the initial value $\varphi>u^{\min}_+$, then $T_t^- \varphi$ converges to $u^{\max}_-$ uniformly on $M$ as $t\to +\infty$;\\ If there is $x_0\in M$ such that $\varphi(x_0)<u^{\min}_+$, then $T_t^- \varphi$ converges to $-\infty$ uniformly on $M$ as $t\to +\infty$.

  \item[(2)] If the initial value $\varphi<u^{\max}_-$, then $T_t^+ \varphi $ converges to  $u^{\min}_+$ uniformly on $M$ as $t\to +\infty$;\\If there is $x_0\in M$ such that $\varphi(x_0)>u^{\max}_-$, then $T_t^+\varphi$ converges to $+\infty$ uniformly on $M$ as $t\to +\infty$.
\end{itemize}	
\end{proposition}

\begin{proof}
Item (1) follows directly from \cite[Theorem 1-3]{NW}. Here we show how to deduce item (2) from (1): As is highlighted in Proposition \ref{sol-HJe} and \ref{weak-kam-vis}, there is a symmetric relationship for solutions of \eqref{HJe} and \eqref{HJs} with Hamiltonian $H$ and $\breve{H}$. Since the map $H\mapsto\breve{H}$ is an involution, we can replace $H$ by $\breve{H}$ in Proposition \ref{sol-HJe} and \ref{weak-kam-vis} to get \begin{enumerate}[(i)]
  \item $-T_t^+(-\varphi)/-T_t^-(-\varphi)$ are the backward/forward solution semigroup associated to \eqref{HJe} with $G=\breve{H}$.

  \item $u_-/u_+$ is a backward/forward weak KAM solution for \eqref{HJs} with $G=H$ if and only if $-u_-/-u_+$ is a forward/backward weak KAM solution for \eqref{HJs} with $G=\breve{H}$.
\end{enumerate}
Observe that if $H$ is of the form \ref{model} and satisfies ($\clubsuit$) and ($\blacklozenge$), so does $\breve{H}$. Thus the maximal (resp. minimal) backward (resp. forward) weak KAM solution for \eqref{HJs} with $G=\breve{H}$ is $-u^{\min}_+$ (resp. $-u^{\max}
_-$). We apply (1) to $\breve{H}$ to obtain that for any $-\varphi>-u^{\max}_-$,
$$
\lim_{t\to +\infty}-T_t^+\varphi=\lim_{t\to +\infty}-T_t^+(-(-\varphi))=-u^{\min}_+\Leftrightarrow\lim_{t\to +\infty}T_t^+ \varphi =u^{\min}_+.
$$
This finishes the proof.
\end{proof}

As a corollary of our construction and the above results, we have
\begin{proposition}\label{max/min-sol-ty}
$u_-^{\max}=\bar{u}_-, u^{\min}_+=\underline{u}_+$.
\end{proposition}

\begin{proof}
From Proposition \ref{prop:Large-time behavior1} (1), $\varphi\geqslant u^{\min}_+$, otherwise it contradicts \eqref{principle-sol1}. By strict monotonicity of the action of $T^-_t$ on $\varphi$, one must have $\bar u_-> u^{\min}_+$. We apply Proposition \ref{prop:Large-time behavior1} (1) again to deduce that
\[
\bar{u}_-=\lim_{t\to \infty}T_t^-\bar{u}_-=u_-^{\max}.
\]
In a similar way, $\underline{u}_+= u^{\min}_+$.
\end{proof}

\subsection{Structure of Ma\~{n}\'{e} set and asymptotic behavior of action minimizing orbits}
Concerning the topology and the location of the Ma\~{n}\'{e} set for our model system, we obtain
\begin{theorem}\label{mane-loc}
The limits
\begin{equation}\label{principle-sol2}
\underline{u}_-:=\lim_{t\rightarrow\infty}T^-_t\underline{u}_+,\quad\bar{u}_+:=\lim_{t\rightarrow\infty}T^+_t\bar{u}_-.
\end{equation}
exists and the convergence is uniform. The Ma\~{n}\'{e} set $\widetilde{\mathcal{N}}$ associated to the system \eqref{ch} is compact and
\begin{equation}\label{eq:mane-loc}
\widetilde{\mathcal{N}}_{\bar{u}_-}\sqcup\widetilde{\mathcal{N}}_{\underline{u}_+}\subset\widetilde{\mathcal{N}}\subset\{(x,u,p)\in J^1(M,\R):\underline{u}_-(x)\leq u\leq\bar{u}_+(x)\}.
\end{equation}
\end{theorem}

\begin{remark}
It should be pointed out that $\underline{u}_-(x)\leq\bar{u}_+(x)$ \textbf{does not hold} throughout $M$.
\end{remark}

\begin{proof}
By Proposition \ref{A6} and the fact that $\underline{u}_+,\bar{u}_-$ are subsolutions, $T^-_t\underline{u}_+, T^+_t\bar{u}_-$  are monotone in $t$. By Proposition \ref{prop-sg} (1), we have the two-sided bounds $\underline{u}_+\leqslant T^-_t\underline{u}_+\leqslant\bar{u}_-, \underline{u}_+\leqslant T^+_t\bar{u}_-\leqslant\bar{u}_-$. Now the limits in \eqref{principle-sol2} exist pointwise and the uniform convergence is guaranteed by Dini's theorem.

\vspace{0.5em}
Observe that the first inclusion in \eqref{eq:mane-loc} is a direct consequence of Theorem \ref{mane-nonempty-decomposition} and the compactness of $\widetilde{\cN}$ is implied by the second inclusion and the fact that $\widetilde{\cN}\subset H^{-1}(0)$.

\vspace{0.5em}
For any $u_-\in\cS_-$, the limit $\lim_{t\rightarrow\infty}T^+_tu_-$ exists and belongs to $\cS_+$. By the minimality of $\underline{u}_+$ and Proposition \ref{A6}, $\underline{u}_+\leqslant\lim_{t\rightarrow\infty}T^+_tu_-\leqslant u_-$. Then $\underline{u}_-:=\lim_{t\rightarrow\infty}T^-_t\underline{u}_+\leqslant\lim_{t\rightarrow\infty}T^-_tu_-=u_-$, meaning that $\underline{u}_-$ is \textbf{the minimal element} of $(\cS_-,\preceq)$. In a similar way, $\bar{u}_+$ is \textbf{the maximal element} of $(\cS_+,\preceq)$. We complete the proof by employing Theorem \ref{mane-nonempty-decomposition} again.
\end{proof}

The next theorem is a direct application of \eqref{order-toymodel}, Proposition \ref{prop:Large-time behavior1} and our main results to the dynamics of $\Phi^t_H$ for our model system.
\begin{theorem}\label{hor}
Assume  $\varphi\in C^2(M,\R)$.
\begin{enumerate}[(1)]
%

  \item If $\min_{x\in M}\{\varphi(x)-\underline{u}_+(x)\}>0$, then there exists $z\in J^1_\varphi$ such that $\omega(z)\subset \widetilde{\mathcal{N}}_{ \bar{u}_-}$.
  \item If $\max_{x\in M}\{\varphi(x)-\bar{u}_-(x)\}<0$, then there exists $z\in J^1_\varphi$ such that $\alpha(z)\subset\widetilde{\mathcal{N}}_{\underline{u}_+}$.
  \item There exists $z\in J^1(M,\R)$ such that
       \[
       \alpha (z) \subset \widetilde{\mathcal{N}}_{\underline{u}_+},\quad\omega(z)\subset \widetilde{\mathcal{N}}_{ \bar{u}_-}.
       \]
\end{enumerate}
\end{theorem}

In fact, if we focus on the action minimizing orbits, then there is a complete classification of limit sets depending only on the front projection of the initial point.
\begin{TheoremX}\label{thm3:dy-minimizer}
For each $z_0=(x_0,u_0,p_0)\in\widetilde{\cG}$, set $\Phi^t_H(z_0)=(x(t),u(t),p(t))$, then
\begin{itemize}
  \item[(1)] If  $u_0\in(\bar{u}_-(x_0),+\infty)$, then $\omega (z_0)\subset\widetilde{\mathcal{N}}_{ \bar{u}_-}$ and $\lim_{t\rightarrow-\infty}u(t)=+\infty$;
  \item[(2)] If  $u_0=\bar{u}_-(x_0)$, then $\alpha(z_0)\cup\omega(z_0)\subset\widetilde{\mathcal{N}}_{\bar{u}_-}$;
  \item[(3)] If  $u_0\in(\underline{u}_+(x_0),\bar{u}_-(x_0))$, then $\alpha(z_0)\subset\widetilde{\mathcal{N}}_{\underline{u}_+}, \omega (z_0)\subset\widetilde{\mathcal{N}}_{ \bar{u}_-}$;
  \item[(4)] If  $u_0=\underline{u}_+(x_0)$, then $\alpha(z_0)\cup\omega(z_0)\subset\widetilde{\mathcal{N}}_{\underline{u}_+}$;
  \item[(5)] If  $u_0\in(-\infty,\underline{u}_+(x_0))$, then $\alpha(z_0)\in\widetilde{\mathcal{N}}_{\underline{u}_+}$ and $\lim_{t\rightarrow+\infty}u(t)=-\infty$.
\end{itemize}
\end{TheoremX}

\vspace{0.1em}
\begin{proof}
We shall only concern the asymptotic behavior of $\Phi^t_H(z_0)$ when $t$ goes to $+\infty$. The proof for the other time direction is completely similar.

\vspace{0.5em}
\textit{Case $(i)$: $u_0>\underline{u}_+(x_0)$}. Notice that since $z_0\in\widetilde{\cG}$, for any $s>0$,
\[
u(t+s)=h_{x_0,u_0}(x(t+s),t+s)=T_t^-h_{x_0,u_0}(x(t+s),s),\quad \forall t\in[0,+\infty),
\]
where the second equality follows from Proposition \ref{prop-sg} (5). By the above equation and Proposition \ref{omega-mane} with $\varphi=h_{x_0,u_0}(\cdot,s)$, we conclude that if
\begin{equation}\label{eq1:proof-C}
\lim_{t\rightarrow+\infty}h_{x_0,u_0}(x,t+s)=\lim_{t\rightarrow+\infty}T^-_t h_{x_0,u_0}(x,s)=\bar{u}_-,
\end{equation}
then $\omega(z_0)\in\widetilde{\cN}_{\bar{u}_-}$. Now we show that \eqref{eq1:proof-C} holds: set
$$
\psi(x):=\underline{u}_+(x)+u_0-\underline{u}_+(x_0) >\underline{u}_+(x), \quad \forall x\in M
$$
Applying Proposition \ref{prop:Large-time behavior1} and \ref{max/min-sol-ty}, $T_t^- \psi$ converges to $\bar u_-$ uniformly on $M$ as $t\to +\infty$. Thus, $T_t^-\psi>\underline{u}_+$ when $t$ is large enough and for such a $t>0$,
$$
\underline{u}_+(x)<T_t^-\psi(x)=\inf_{y\in M} h_{y,\psi(y)}(x,t)\leqslant h_{x_0,\psi(x_0)}(x,t)=h_{x_0,u_0}(x,t),
$$
implying \eqref{eq1:proof-C}. Since the initial point $z_0$ can be chosen arbitrarily on the orbit $\Phi^t_H(z_0)$, we also deduce that in this case, \begin{equation}\label{case1}
u(t)>\underline{u}_+(x(t))\quad\textbf{for any}\quad t\in\R.
\end{equation}
	
\vspace{0.5em}
\textit{Case $(ii): u_0<\underline{u}_+(x_0)$}. We show that there exists $s_0>0,x\in M$ such that
\[
h_{x_0,u_0}(x,s_0)<\underline{u}_+(x).
\]
Fixing $x_1\in M$, let $\gamma:[0,1]\to M$ be the minimizer of $h_{x_0,u_0}(x_1,1)$, then by Remark \ref{minimizer-orbit},
$$
\lim_{s\to 0^+}[h_{x_0,u_0}(\gamma(s),s)-\underline{u}_+(\gamma(s))]=u_0-\underline{u}_+(x_0)<0.
$$
Thus a sufficient small $s_0\in(0,1]$ and $x=\gamma(s_0)$ shall do the job. Choosing $\varphi=h_{x_0,u_0}(\cdot,s_0)$ and applying Proposition \ref{prop:Large-time behavior1} and \ref{max/min-sol-ty}, $T_t^-\varphi$ tends to $-\infty$ uniformly on $M$ as $t\to +\infty$. Thus
\[
\lim_{t\rightarrow+\infty}u(t)=\lim_{t\rightarrow+\infty}T^-_t\varphi(x(t))=-\infty.
\]
Since the initial point $z_0$ can be chosen arbitrarily on the orbit $\Phi^t_H(z_0)$, we also deduce that in this case,
\begin{equation}\label{case2}
u(t)<\underline{u}_+(x(t))\quad\textbf{for any}\quad t\in\R.
\end{equation}

\vspace{0.5em}
\textit{Case $(iii)$}: Excluding \eqref{case1} and \eqref{case2}, we have $u(t)=\underline{u}_+(x(t)), t\in\R$. Equivalently,
\[
\pi\,\{\Phi^t_H(z_0):t\in\R\}\subset J^0_{\underline{u}_+}.
\]
We invoke Lemma \ref{proof-A2-p} and Remark \ref{proof-A2-p'} to conclude that $z(t)\in J^1_{\underline{u}_+}$ for any $t\in\R$. Then $\omega(z_0)$ is a non-empty subset of $\cN_{\underline{u}_+}$. This completes the proof.
\end{proof}

\subsection{Heteroclinic orbits between two Ma\~{n}\'{e} slices with non-zero energy}
As the last part of this section, we show that heteroclinic orbits constructed in (B2) can lies outside $H^{-1}(0)$ by analysing the
\begin{example}\cite[Example 3.1]{JYZ}\label{non-zero energy}
Consider the contact Hamiltonian
\begin{equation}\label{eq:ex1}
H(x,u,p)=p^2+ \sin x\cdot u-\frac{1}{4} , \quad x\in\mathbb{S}=[-\pi,\pi].
\end{equation}
Then $H$ belongs to the category of model Hamiltonian \eqref{model} and satisfies ($\clubsuit$). According to \cite[Proposition 3.2]{JYZ}, $c(H)<0$. In \cite[Lemma 3.11, Theorem 3.14]{JYZ}, we showed that the corresponding Hamilton-Jacobi equation admits \textbf{exactly} two solutions $\underline{u}_-\leqslant\bar{u}_-$ and $\bar{u}_-$ is \textbf{everywhere positive} on $\mathbb{S}$.
\end{example}
Setting $\mathcal{Z}_0=H^{-1}(0)$ to be the two dimensional energy shell, it is easy to check that
\begin{itemize}
  \item $\cN_{\underline{u}_+}=\{\underline{z}=  (-\frac{\pi}{2},-\frac{1}{4},0) \} , \cN_{\bar{u}_-}=\{\bar{z}=    (\frac{\pi}{2},\frac{1}{4},0)\}$ are two fixed point of $\Phi^t_H$ on $\mathcal{Z}_0$. Moreover, linearization of $\Phi^t_H|_{\mathcal{Z}_0}$, the restriction of $\Phi^t_H$ onto $\mathcal{Z}_0$, at $\underline{z},\bar{z}$ shows that both fixed points are hyperbolic and admit one dimensional unstable (and stable) manifolds on $\mathcal{Z}_0$.

  \item the continuous function
        $$
        \varphi(x):= \begin{cases}
	    \frac{1}{2}\sin x+ \frac{1}{4} , \quad & x\in [-\pi,0],\\
	    \frac{1}{4}, \quad & x\in [0,\pi].
        \end{cases}
        $$
        is a subsolution of \eqref{eq:ex1} with $\varphi(\frac{\pi}{2})=\frac{1}{4} $ and $\varphi(-\frac{\pi}{2})=-\frac{1}{4}$. From Proposition \ref{prop-sg} (2), $T_t^-\varphi\geqslant\varphi$ and for any $t\geqslant1,
        T_t^- \varphi(x)\leqslant h_{-\frac{\pi}{2},-\frac{1}{4}}(x,t)\leqslant h_{-\frac{\pi}{2},h_{-\frac{\pi}{2},-\frac{1}{4}}(-\frac{\pi}{2},t-1)}(x,1)\leqslant h_{-\frac{\pi}{2},-\frac{1}{4}}(x,1)$. In particular, $T_t^- \varphi(-\frac{\pi}{2})\leqslant h_{-\frac{\pi}{2},-\frac{1}{4}}(-\frac{\pi}{2},1)\leqslant-\frac{1}{4}<0$. This implies that $\lim_{t\to +\infty} T_t^- \varphi =\underline{u}_-\in\mathcal{S}^-$.
\end{itemize}
By (B2) or Theorem \ref{hor}, there is $z_0\in J^1(M,\R)$ such that $\alpha(z_0)=\{\underline{z}\},\omega(z_0)=\{\bar{z}\}$. Thus the goal of this section is reduced to prove

\vspace{0.5em} 
\textit{Claim:} Any heteroclinic orbit connecting $\underline{z}$ to $\bar{z}$ can not lie on $\mathcal{Z}_0$.

\vspace{0.5em}
We argue by contradiction. Assume there is a heteroclinic orbit $z(t)=(x(t),u(t),p(t)):\R\rightarrow\mathcal{Z}_0$ with $\lim_{t\rightarrow-\infty}z(t)=\underline{z}, \lim_{t\rightarrow+\infty}z(t)=\bar{z}$. By definition, $z(t)$ is contained in $W^{u}(\underline{z})$, the unstable manifold of $\underline{z}$. Due to the the fact that, \textbf{near $\underline{z}$}, $W^u(\underline{z})$ is a graph over $\rho$, we use $J^0_{\phi_0}$ to denote $\pi(W^u(\underline{z}))$ within a neighborhood of $(-\frac{\pi}{2},-\frac{1}{4})$, where $\phi_0:[-\frac{\pi}{2}-\sigma_0, -\frac{\pi}{2}+\sigma_0]\rightarrow\R$ is a continuous function and $\sigma_0>0$. Without loss of generality, we assume $p(t)>0$ for $t$ small enough (choosing one branch of unstable manifold). Then \eqref{ch} implies
\begin{equation}\label{eq:zero-energy}
\dot{x}(t)=p(t)\quad \dot{u}(t)=2|p(t)|^2,\quad\quad t\in\R
\end{equation}
from which we deduce that $\frac{1}{4}-\sin x(t)\cdot u(t)>0$ if $x(t)\leqslant0$. The fact $z(t)\subset\mathcal{Z}_0$ translates into
\[
|p(t)|^2=\frac{1}{4}-\sin x(t)\cdot u(t),
\]
showing that $p(t)>0$ if $x(t)\leqslant0$. Combining the above discussions,
\begin{enumerate}[(i)]
  \item $x(t),u(t)$ is strictly increasing if $x(t)\leqslant0$. Equivalently, $\sigma_0>\frac{\pi}{2}$ and $\phi_0|_{[-\frac{\pi}{2},0]}$ is strictly increasing.

  \item there is $\mathcal{T}\in\R$ such that $x(\mathcal{T})=0$ and $x(t)\in(-\frac{\pi}{2},0)$ for $t\in(-\infty,\mathcal{T})$. As a result, \begin{equation}\label{non-zero-proof}
      u(\mathcal{T})=\phi_0(x(\mathcal{T})).
      \end{equation}
\end{enumerate}
We showed in \cite[Theorem 3.14]{JYZ} that $\underline{u}_-(0)\leqslant\phi_0(0)$ and it follows from \eqref{non-zero-proof} that
$$
u(\mathcal{T})=\phi_0(0)\geqslant u_-(0)\geqslant\varphi(0)=\frac{1}{4}.
$$
Due to \eqref{eq:zero-energy} and the fact that $p(\mathcal{T})=\frac{1}{2}>0$, we conclude that $u(t)>\frac{1}{4}$ and is increasing for $t\in(\mathcal{T},+\infty)$, contradicting the equation $\lim_{t\rightarrow+\infty}u(t)=\frac{1}{4}$.

\begin{remark}
	With the aid of numerical method, we depict two orbits  to \eqref{ch} below in Figure \ref{fig1}. The green surface denote  the  two dimensional null energy surface $\mathcal{Z}_0$.
	The red orbit denotes the  heteroclinic orbit between two hyperbolic fixed points  $\underline{z} $ , $ \bar{z}$.   
		The blue orbit denotes the one dimensional unstable   manifolds $W^u(\underline{z})$ on $\mathcal{Z}_0$.
	The numerical results fit well into our analysis, especially,  (B2) and Example \ref{non-zero energy}.
\end{remark}
\begin{figure}[h]
\begin{center}
\includegraphics[width=13cm]{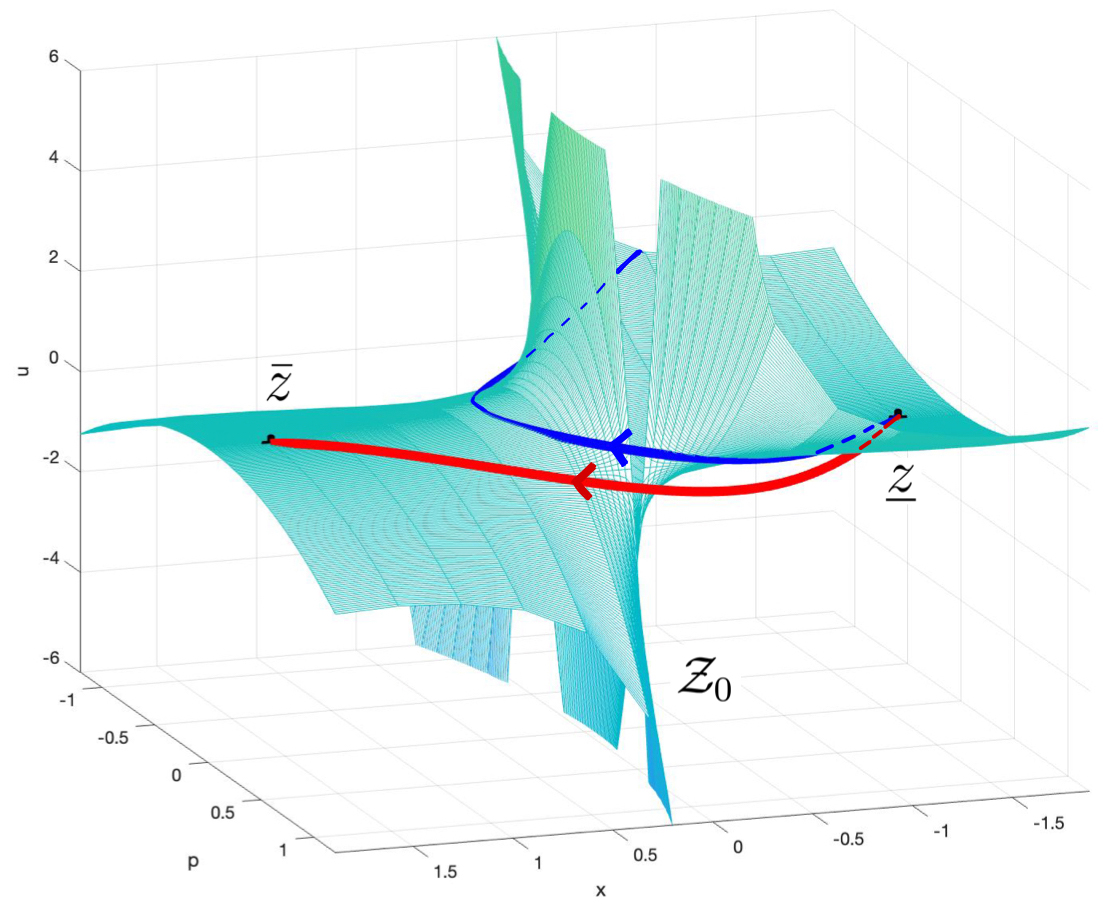}
\label{fig1}
\end{center}
\end{figure}
 
\begin{remark}
In Section 5, we have shown that the heteroclinic orbits found by (B2) is action minimizing. However, by the above discussion, these orbits can avoid the null energy level, and thus is not semi-static $($because due to Theorem \ref{mane-nonempty-decomposition}, any semi-static orbit is part of $\widetilde{\mathcal{N}}_{v_+}$ or $\widetilde{\mathcal{N}}_{u_-}$, thus lies on $H^{-1}(0))$.  This fact shows for autonomous contact Hamiltonian systems, the set of action minimizing orbits is in general strictly larger than the set of semi-static orbits. This is different from autonomous Hamiltonian systems, see \cite[Section 5]{Ber} for details.
\end{remark}

\section{Appendix}
The Appendix is divided into two parts: the first part is a completion of the properties of useful tools, namely action functions and solution semigroups, developed in \cite{WWY1}-\cite{WWY3} on the action minimizing method for contact Hamiltonian systems; the second includes a considerably self-contained proof of the decomposition theorem of semi-static curves with emphasis on the construction of certain weak KAM solutions that calibrating the specified semi-static curve.

\subsection{Properties of action functions and solution semigroups}
We collect fundamental properties of the action functions and solution semigroups that are frequently used but not presented in the main body of this article. For details and proofs of these properties, we encourage the reader to consult the original paper \cite{WWY2}.

\begin{proposition}{\cite{WWY2}}\label{fundamental-prop}
The backward and forward action functions satisfy
\begin{enumerate}
    \item[(1)]\textbf{($u_0$-monotonicity)} Given $x_0\in M, u_1<u_2\in\R$,\,\,for all $(x,t)\in M\times(0,\infty)$,
    \[
	h_{x_0,u_1}(x,t)< h_{x_0,u_2}(x,t),\quad h^{x_0,u_1}(x,t)< h^{x_0,u_2}(x,t).
    \]

	\item[(2)]\textbf{(Markov property)} Given $(x_0,u_0)\in M\times\R$,\,\,for all $t,\tau>0$ and $  x\in M$,
	\begin{equation}\label{markov}
    \begin{split}
	&h_{x_0,u_0}(x, t+\tau)=\inf_{x_1\in M}h_{x_1,h_{x_0,u_0}(x_1,t)}(  x,\tau),\\
    &h^{x_0,u_0}(x, t+\tau)=\sup_{x_1\in M}h^{x_1,h^{x_0,u_0}(x_1,t)}(  x,\tau).
    \end{split}
	\end{equation}
	Moreover, the infimum is attained at $x_1$ if and only if there exists a $C^1$ minimizer $\gamma$ of $h_{x_0,u_0}(x,t+\tau)$ with $ \gamma(t)=x_1$, the supremum is attained at $x_1$ if and only if there exists a $C^1$ minimizer $\gamma$ of $h^{x_0,u_0}(x,t+\tau)$ with $ \gamma(t)=x_1$.
	
    \item[(3)]\textbf{(Reversibility)} Given $x_0, x\in M$ and $t>0$,
        \[
        u=h_{x_0,u_0}(x,t)\quad\text{if and only if}\quad u_0=h^{x,u}(x_0,t).
        \]

	\item[(4)]\textbf{(Lipschitz continuity)} The functions
    \[
    (x_0,u_0,x,t)\mapsto h_{x_0,u_0}(x,t),\quad (x_0,u_0,x,t)\mapsto h^{x_0,u_0}(x,t)
    \]
    are locally Lipschitz continuous on the domain $M\times\R\times M\times(0,+\infty)$.
\end{enumerate}
\end{proposition}

\begin{proposition}\label{prop-sg}\cite[Proposition 4.3]{WWY2}
Two families of operator $\{T^{\pm}_t\}_{t\geqslant0}$ defined above satisfy
\begin{enumerate}
    \item[(1)]\textbf{($\varphi$-monotonicity)} For initial data $\varphi,\psi\in C(M,\R)$ with $\varphi<\psi$ (resp. $\varphi\leqslant\psi$) on $M$, then for all $x\in M$,
        \begin{equation}\label{mono-sg}
	    T^{\pm}_t\varphi(x)<T^{\pm}_t\psi(x),\quad\text{resp.}\,\,(T^{\pm}_t\varphi(x)\leqslant T^{\pm}_t\psi(x)).
        \end{equation}

	\item[(2)]\textbf{($t$-monotonicity)} Assume $v\in C(M,\R)$ is a subsolution (resp. strict subsolution)  to \eqref{hj}. Then for any $x\in M, t\geqslant0$ (resp. $t>0$),
        \begin{equation}\label{mono-t}
        v(x)\leqslant\,\,(\text{resp}. <)\,\,T^{-}_t v(x),\quad v(x)\geqslant\,\,(\text{resp}. >)\,\,T^{+}_t v(x).
	    \end{equation}

	\item[(3)]\textbf{(Continuity 1)} For any $(x,t)\in M\times(0,+\infty)$, the functions
     \[
     (x,t)\mapsto T^{\pm}_t\varphi(x),
     \]
     are locally Lipschitz continuous and $\lim_{t\rightarrow0^+}T^{\pm}_t\varphi(x)=\varphi(x)$ for all $x\in M$.

    \item[(4)]\textbf{(Continuity 2)} For any $\varphi,\psi\in C(M,\R)$ and $t\geqslant0$,
     \[
     \|T^{\pm}_t\varphi-T^{\pm}_t\psi\|_\infty\leqslant e^{\lambda t}\|\varphi-\psi\|_\infty.
     \]
     As a result, the map $v\mapsto T^{\pm}_t v$ are continuous with respect to $\|\cdot\|_\infty$.

    \item[(5)]\textbf{(Relation with action functions)} Given $(x_0,u_0)\in J^0(M,\R)$, for any $s,t>0$ and $x\in M$,
     \[
     T^{-}_t h_{x_0,u_0}(x,s)=h_{x_0,u_0}(x,s+t),\quad T^{+}_t h^{x_0,u_0}(x,s)=h^{x_0,u_0}(x,s+t).
     \]
\end{enumerate}
\end{proposition}

\begin{proposition} \label{A6}
\cite[Proposition 10]{W-Y} Let $\varphi \in C(M,\R)$, then
$$T_t^- \circ T_t^+ \varphi \geqslant \varphi ,\quad T_t^+ \circ T_t^- \varphi \leqslant \varphi,  \quad  \forall \ t >0.$$
\end{proposition}
\begin{proof}
	Applying \eqref{eq:Tt-+ rep} and Proposition \ref{fundamental-prop} (1)(3), we have for any $t>0$,
$$
T_t^- \circ T_t^+ \varphi(x)= \inf_{y\in M} h_{y, T_t^+ \varphi(y)} (x,t) \geqslant \inf_{y\in M}  h_{y, h^{x,\varphi(x)}(y,t)}(x,t)=\varphi (x),
$$
$$
T_t^+ \circ T_t^- \varphi(x)= \sup_{y\in M} h^{y, T_t^- \varphi(y)} (x,t) \leqslant \sup_{y\in M}  h^{y, h_{x,\varphi(x)}(y,t)}(x,t)=\varphi (x).
$$
\end{proof}

\subsection{Decomposition of semi-static orbits}
For the readers' convenience, this part is devoted to a proof of Theorem \ref{mane-nonempty-decomposition}, which states that the set of semi-static orbits is non-empty if and only if the set of weak KAM solutions are non-empty and when semi-static orbits exist, they form a closed set and calibrated by some weak KAM solutions. Despite its self-contained elaboration, the constructions in the proof, especially Lemma \ref{Busemann}-\ref{gra-line} and Remark \ref{max-min-cali}, are noteworthy. To begin with, we recall the basic fact
\begin{lemma}\cite[Proposition 3.2]{RWY}\label{0-bound}
Each semi-static orbit $\Phi^t_H(z_0)=(x(t),u(t),p(t)), t\in\R$ is bounded, i.e., there is a constant $B(z_0)>0$, continuously depending on $z_0$, such that $\sup_{t\in\R}(|u(t)|+\|p(t)\|_{x(t)})\leqslant B$.
\end{lemma}


As a direct consequence of Proposition \ref{weak-kam}, we have
\begin{lemma}\cite[Proposition 3.6]{RWY}\label{cali-semi-static}
Assume $u_\pm\in\mathcal{S}_\pm$. Then for each $(u_\pm,L)$-calibrated curve $\gm:\R\rightarrow M$,
\[
(x(t),u(t),p(t)):=(\gm(t),\,\,u_\pm\circ\gm(t),\,\,\frac{\partial L}{\partial\dot{x}}(\gm(t),u_\pm\circ\gm(t)),\dot{\gm}(t)),\,\,t\in\R
\]
is a semi-static orbit.
\end{lemma}


In the following, we show that each semi-static curve serves as a calibrated curve for some weak KAM solution, which reveals a duality between solutions to \eqref{HJs} and global minimizers. This fact is already observed in \cite{Fathi_book}. We start with the construction of Busemann-type weak KAM solutions from semi-staic curves. Let $(x(\cdot), u(\cdot)):\R\rightarrow M\times\R$ be a semi-static curve. Define for any $(x,t)\in M\times\R$,
\begin{equation}\label{def-Busemann}
\begin{split}
U_-(x,t)&:=\inf_{s>0}h_{x(t),u(t)}(x,s),\\
(\text{resp.}\quad U_+(x,t)&:=\sup_{s>0}h^{x(t),u(t)}(x,s)).
\end{split}
\end{equation}

\begin{lemma}\label{Busemann1}
Assume $U_\pm(\cdot,t):M\rightarrow\R$ are defined as in \eqref{def-Busemann}. Then
\begin{enumerate}[(1)]
  \item $U_-$ is monotone increasing and bounded in $t$,
  \item $U_+$ is monotone decreasing and bounded in $t$.
\end{enumerate}
\end{lemma}

\begin{proof}
We shall focus on (1), since (2) can be proved in a similar way. Notice that for each $x\in M$ and $t<t'$,
\begin{align*}
U_-(x,t)=&\,\inf_{s>0}h_{x(t),u(t)}(x,s)=\inf_{s',s''>0}\inf_{y\in M}h_{y,h_{x(t),u(t)}(y,s')}(x,s'')\\
\leqslant&\,\inf_{s',s''>0}h_{x(t'), h_{x(t),u(t)}(x(t'),s')}(x,s'')=\inf_{s''>0}h_{x(t'),u(t')}(x,s'')=U_-(x,t'),
\end{align*}
where the second equality follows from Proposition \ref{fundamental-prop} (2), the third equality follows from Definition \ref{semi-static}. This shows that $U_-(x,t)$ is increasing in $t$.

\vspace{1em}
Applying Lemma \ref{0-bound}, $B:=\sup_{(x,t)\in M\times\R}\,\,\{|h_{x(t),u(t)}(x,1)|, |h^{x(t),u(t)}(x,1)|\}<+\infty$. It follows that
\[
U_-(x,t)\leqslant h_{x(t),u(t)}(x,1)\leqslant B.
\]
By Definition \ref{semi-static} and Proposition \ref{fundamental-prop} (2), for any $t<t'$ and $s>0$ ,
\[
u(t')\leqslant h_{x(t),u(t)}(x(t'),s+1)\leqslant h_{x,h_{x(t),u(t)}(x,s)}(x(t'),1).
\]
Combining Proposition \ref{fundamental-prop} (1) and (3), the above inequality implies
\[
U_-(x,t)=\inf_{s>0}h_{x(t),u(t)}(x,s)\geqslant h^{x(t'),u(t')}(x,1)\geqslant-B.
\]
Thus for any $(x,t)\in M\times\R$,
\begin{equation}\label{Buse-bdd}
|U_-(x,t)|\leqslant B.
\end{equation}
\end{proof}

Next we show that
\begin{lemma}\label{Busemann2}
$\{U_\pm(\cdot,t)\}_{t\in\R}$ is equi-Lipschitz.
\end{lemma}

\begin{proof}
For any $x,x'\in M$ and $t\in\R$,
\begin{align*}
&U_-(x,t)-U_-(x',t)=\inf_{s,s'>0}h_{x(t),u(t)}(x,s+s')-U(x',t)\\
\leqslant&\,\inf_{s,s'>0}h_{x',h_{x(t),u(t)}(x',s')}(x,s)-U_-(x',t)\\
=&\,\inf_{s>0}h_{x',\inf_{s'>0}h_{x(t),u(t)}(x',s')}(x,s)-U_-(x',t)\\
=&\,\inf_{s>0}h_{x',U_-(x',t)}(x,s)-U_-(x',t)\\
\leqslant\,&\kappa(B)\cdot d(x,x').
\end{align*}
Here, the first inequality and second equality follows from Proposition \ref{fundamental-prop} (1) and (2). The last inequality is a consequence of \eqref{Buse-bdd} and the general fact that when $|u|\leqslant B$,
\begin{equation}\label{lip-est}
\inf_{s>0}h_{x',u}(x,s)-u\leqslant h_{x',u}(x,d(x,x'))-u\leqslant\kappa(B)\cdot d(x,x').
\end{equation}
In fact, \eqref{lip-est} holds in the sense that the left hand side maybe $-\infty$, and the second inequality above is deduced by applying in \cite[Lemma 3.1-3.2]{WWY1}. Observing that exchanging the role of $x$ and $x'$ in the computation, we conclude that for any $t\in\R, U_-(\cdot,t)$ is $\kappa(B)$-Lipschitz.
\end{proof}

\begin{lemma}\label{Busemann3}
For any $t<t', U_+(\cdot,t')\leqslant U_-(\cdot,t)$.
\end{lemma}

\begin{proof}
We shall show that for any $s,s'>0$,
\begin{equation}\label{order}
h_{x(t),u(t)}(x,s)\geqslant h^{x(t'),u(t')}(x,s').
\end{equation}
We argue by contradiction, assume $h_{x(t),u(t)}(x,s)<h^{x(t'),u(t')}(x,s')$. Then by Definition \ref{semi-static},
\begin{align*}
u(t')\,\leqslant\,\,&h_{x(t),u(t)}(x(t'),s+s')\\
\leqslant\,\,&h_{x,h_{x(t),u(t)}(x,s)}(x(t'),s')\\
<\,\,&h_{x,h^{x(t'),u(t')}(x,s')}(x(t'),s')=u(t'),
\end{align*}
where the third inequality and the equality follow from Proposition \ref{fundamental-prop} (1) and (3). Thus
\[
U_-(x,t)=\inf_{s>0}h_{x(t),u(t)}(x,s)\geqslant\sup_{s'>0}h^{x(t'),u(t')}(x,s')=U_+(x,t').
\]
\end{proof}


\begin{lemma}\label{Busemann}
The uniform limits
\[
\mathrm{u}_-(x)=\lim_{t\rightarrow-\infty}U_{-}(x,t)=\inf_{t\in\R}U_-(x,t),\quad \mathrm{u}_+(x)=\lim_{t\rightarrow+\infty}U_{+}(x,t)=\sup_{t\in\R}U_+(x,t)
\]
exist. Moreover, $\mathrm{u}_-\in\mathcal{S}_-, \mathrm{u}_+\in\mathcal{S}_+$ and $\mathrm{u}_-\geqslant\mathrm{u}_+$.
\end{lemma}

\begin{proof}
Lemma \ref{Busemann1} and \ref{Busemann2} imply the limits above exist. Dini's theorem on monotone convergence implies that the limits are uniform. We shall show $\mathrm{u}_-\in\mathcal{S}_-$, the same proof apply to $\mathrm{u}_+$ with few changes.

\vspace{0.5em}
Notice that for $\tau>0$, we have
\begin{equation}\label{eq:1}
T^-_{\tau}U_-(x,t)=T^-_\tau\big[\inf_{s>0} h_{x(t),u(t)}(x,s)\big]=\inf_{s>0}T^-_\tau h_{x(t),u(t)}(x,s)=\inf_{s>0}h_{x(t),u(t)}(x,s+\tau),
\end{equation}
where the third equality follows from Proposition \ref{prop-sg} (5). It follows from Proposition \ref{fundamental-prop} (2) that
\[
U_-(x,t)\leqslant\inf_{s>\tau}h_{x(t),u(t)}(x,s)=\inf_{s>0}h_{x(t),u(t)}(x,s+\tau)\leqslant\inf_{s>0}h_{x(t+\tau),u(t+\tau)}(x,s)=U_-(x,t+\tau).
\]
Combining the above inequality and \eqref{eq:1}, for any $\tau>0$,
\begin{equation}\label{ieq:1}
U_-(x,t)\leqslant T^-_{\tau}U(x,t)\leqslant U_-(x,t+\tau)
\end{equation}
Now sending $t$ to $-\infty$ on both sides of \eqref{ieq:1} and applying Proposition \ref{prop-sg} (4), we obtain that $\bar{u}_-\in\mathcal{S}_-$.

\vspace{0.5em}
Finally, using Lemma \ref{Busemann1} and \ref{Busemann3},
$$
\mathrm{u}_-(x)=\lim_{t\rightarrow-\infty}U_-(x,t)=\inf_{t<0}U_-(x,t)\geqslant\sup_{t'>0}U_+(x,t')=\lim_{t'\rightarrow\infty}U_+(x,t')=\mathrm{u}_+(x).
$$
\end{proof}

Now it is easy to see that each semi-static curve are calibrated curve of a pair of weak KAM solutions.
\begin{lemma}\label{gra-line}
For $v=\mathrm{u}_-$ or $\mathrm{u}_+, v(x(t))=u(t)$ for each $t\in\R$. In particular, for any $t<t'$, for
\begin{equation}\label{semi-static-cali}
v(x(t'))-v(x(t))=\int^{t'}_{t}L(x(s),v(x(s)),\dot{x}(s))\ ds.
\end{equation}
\end{lemma}

\begin{proof}
As before, we shall only show $\mathrm{u}_-(x(t))=u(t)$ for each $t\in\R$. For each $t\in\R$,
\begin{align*}
&\mathrm{u}_-(x(t))=\lim_{\tau\rightarrow-\infty}U_-(x(t),\tau)\\
=&\,\lim_{\tau\rightarrow-\infty}\inf_{s>0}h_{x(\tau),u(\tau)}(x(t),s)\\
=&\,\lim_{\tau<t,\tau\rightarrow-\infty}\inf_{s>0}h_{x(\tau),u(\tau)}(x(t),s)\\
=&\,\lim_{\tau<t,\tau\rightarrow-\infty}u(t)=u(t),
\end{align*}
where the fourth equality uses Definition \ref{semi-static}. Since $(x(t), u(t))$ is a global action minimizer, for $t<t'$,
\begin{align*}
&\mathrm{u}_-(x(t'))=u(t')=h_{x(t),u(t)}(x(t'),t'-t)\\
=\,&u(t)+\int^{t'-t}_0 L(x(t+s),h_{x(t),u(t)}(x(t+s),s),\dot{x}(t+s)) ds\\
=\,&u(t)+\int^{t'}_t L(x(s),u(s),\dot{x}(s)) ds\\
=\,&\mathrm{u}_-(x(t))+\int^{t'}_t L(x(s),\mathrm{u}_-(x(s)),\dot{x}(s)) ds.
\end{align*}
Here, the second equality follows from Proposition \ref{prop-act-minimizer} (3), which implies that $x(t+\cdot)|_{[0,t'-t]}$ is the unique minimizer of $h_{x(t),u(t)}(x(t'),t'-t)$.
\end{proof}

\begin{remark}\label{max-min-cali}
It is worth mentioning that
\begin{equation}\label{extremal}
\begin{split}
\mathrm{u}_-(x)=\sup_{u_-\in\mathcal{S}_-}\{u_-(x)\,:\,u_-(x(t))=u(t),\quad t\in\R\},\\
\mathrm{u}_+(x)=\inf_{u_+\in\mathcal{S}_+}\{u_-(x)\,:\,u_+(x(t))=u(t),\quad t\in\R\}.
\end{split}
\end{equation}
To show this, it is necessary to see that, if $u_-\in\mathcal{S}_-$ and $u_-(x(t))=u(t)$, then for any $t\in\R, s>0$,
\begin{equation*}
\begin{split}
u_-(x)=T^-_s u_-(x)=\inf_{y\in M}h_{y,u_-(y)}(x,s)\leqslant h_{x(t),u_-(x(t))}(x,s)=h_{x(t),u(t)}(x,s),\\
u_+(x)=T^+_s u_+(x)=\sup_{y\in M}h^{y,u_+(y)}(x,s)\geqslant h^{x(t),u_+(x(t))}(x,s)=h^{x(t),u(t)}(x,s).
\end{split}
\end{equation*}
\end{remark}

\textit{Proof of Theorem \ref{mane-nonempty-decomposition}:}

\vspace{0.5em}
By Lemma \ref{Busemann} and \ref{gra-line}, for any $z_0\in\tilde{\mathcal{N}}\neq\emptyset$, we can construct $\mathrm{u}_-\in\mathcal{S}_-$ such that $z_0\in\widetilde{\mathcal{N}}_{u_-}$. This implies $\mathcal{S}_-\neq\emptyset$ and $\widetilde{\mathcal{N}}\subset\bigcup_{u_-\in\mathcal{S}_-}\widetilde{\mathcal{N}}_{u_-}$. Conversely, if $\mathcal{S}_-\neq\emptyset$, we choose $u_-\in S_-$ so that $\widetilde{\mathcal{N}}_{u_-}\neq\emptyset$. For any $z_0\in\widetilde{\mathcal{N}}_{u_-}, x(t)=\pi_M\circ\Phi^t_H z_0$ is $(u_-,L)$-calibrated. By Lemma \ref{cali-semi-static}, $\Phi^t_H z_0:\R\rightarrow J^1(M,\R)$ is semi-static and $\widetilde{\mathcal{N}}_{u_-}\subset\widetilde{\mathcal{N}}\neq\emptyset$. The first equality in \eqref{decom-mane} is proved. In the same way, one can show the equivalence between $\tilde{\mathcal{N}}\neq\emptyset$ and $\mathcal{S}_+\neq\emptyset$ as well as the second equality in \eqref{decom-mane}.

\vspace{0.5em}

To show the closedness of $\widetilde{\mathcal{N}}$, let $(x_n(\cdot),u_n(\cdot),p_n(\cdot)):\R\rightarrow J^1(M,\R), n\geqslant1$ be a sequence of semi-static orbits with $(x_n(0),u_n(0),p_n(0))$ converging to $(x_0,u_0,p_0)$. By Lemma \ref{0-bound}, we can assume that, up to a subsequence, $(x_n(\cdot),u_n(\cdot))$ converges uniformly on compact intervals of $\R$ to some $(x(\cdot),u(\cdot)):\R\rightarrow J^0(M,\R)$ with $(x(0),u(0))=(x_0,u_0)$. By Definition \ref{semi-static} and Proposition \ref{fundamental-prop} (4), for any $a<b$ and $s>0$,
\begin{align*}
&\,h_{x(a),u(a)}(x(b),b-a)=\lim_{n\rightarrow\infty}h_{x_n(a),u_n(a)}(x_n(b),b-a)\\
\leqslant&\,\lim_{n\rightarrow\infty}h_{x_n(a),u_n(a)}(x_n(b),s)=h_{x(a),u(a)}(x(b),s),
\end{align*}
or equivalently, \eqref{semi-static curve} holds for any $a<b$. Thus $(x(\cdot),u(\cdot))$ is semi-static and $(\bar{x}_0,\bar{u}_0)\in\mathcal{N}$.
\qed

\section*{Acknowledgments}

L.Jin is partly supported by by the National Key R\&D Program of China 2021YFA1001600. All of the author are supported in part by the National Natural Science Foundation of China (Grant No. 12171096). L.Jin is also supported in part by the NSFC (Grant No. 12371186). J.Yan is also supported in part by the NSFC (Grant No. 12231010). K.Zhao is also supported by   National Natural Science Foundation of China (Grant No. 12301233). The first author would like to thank Dr. S.Suhr for his kind invitation and RUB (Ruhr-Universit\"{a}t Bochum) for its hospitality, where part of this work is done.

\section*{Declaration of competing interest}
There are no competing interests.

\end{document}